\documentclass[twoside,11pt]{article}

\usepackage[preprint]{jmlr2e}

\usepackage{framed}

\usepackage{comment}
\usepackage{xargs}                      
\usepackage[pdftex,dvipsnames]{xcolor}
\definecolor{tblue}{RGB}{0,112,192}
\definecolor{fgreen}{RGB}{34,139,34}
\usepackage{microtype}
\usepackage{graphicx}
\usepackage{subfigure}
\usepackage{hyperref}       
\usepackage{booktabs} 
\usepackage{amsfonts}  
\usepackage{amssymb}
\usepackage{amsmath}
\usepackage{nicefrac}       
 \usepackage{fullpage}
\usepackage{sidecap}
\usepackage{subcaption}
\usepackage{pgf, tikz}
\usepackage{pgfplots}
\usepgfplotslibrary{fillbetween}
\usetikzlibrary{intersections}
\usepackage{diagbox}

\newcommand{\Value}[2]{V_{#1}^{#2}} 

\newcommand{\pik}{\pi^{(k)}}
\newcommand{\pikp}{\pi^{(k+1)}}

\newcommand{\Pt}{\lambda^{(m)}}
\newcommand{\Pk}{Q^{(k)}}
\newcommand{\Ptp}{\lambda^{(m+1)}}

\newcommand{\mb}{\mathbb}
\newcommand{\mc}{\mathcal}

\newcommand{\cl}{\mathrm{cl}}
\newcommand{\tb}{\theta}
\newcommand{\te}{\xi_{0}}

\newcommand{\pib}{\pi_0}

\newcommand{\pie}{\pi}

\newcommand{\indic}[1]{\mathsf{1}_{#1}}

\newcommand{\bias}[2]{H^{#1}_{#2}}

\DeclareMathOperator{\supp}{supp}

\newcommand{\Dc}[2]{\mathsf{D_{mc}}({#1}\|{#2})}

\newcommand{\Lc}[2]{\mathsf{D_{mdp}}({#1}\|{#2})}

\newcommand{\D}[2]{\mathsf{D}({#1}\|{#2})}

\newcommand{\thetahat}{\widehat \theta_{T}}
\newcommand{\qhat}{\widehat\xi_T}

\newcommand{\Vhat}{\widehat V}
\usepackage[british]{babel}
\usepackage{hyperref}  
\usepackage{enumerate}
\usepackage{enumitem}

\newcommand{\fe}{f_{\pi}}
\newcommand{\fpi}{f_{\pi}}

\newcommand{\fpistar}{f_{\pistar(\qhat)}}

\usepackage{algorithm}
\usepackage{algorithmic}

\usetikzlibrary{pgfplots.groupplots}
\usepgfplotslibrary{fillbetween} 

\newcommand{\lsc}{\overline{V}}
\newcommand{\pihat}{\widehat{\pi}}
\newcommand{\pistar}{\pi_\rho}

\DeclareMathOperator*{\argmax}{argmax}

\allowdisplaybreaks

\firstpageno{1}

\begin{document}

\title{Towards Optimal Offline Reinforcement Learning}

\author{\name Mengmeng Li \email mengmeng.li@epfl.ch \\
        \name Daniel Kuhn \email daniel.kuhn@epfl.ch \\
       \addr Risk Analytics and Optimization Chair\\
       EPFL, Switzerland
       \AND
       \name Tobias Sutter \email tobias.sutter@uni-konstanz.de \\
       \addr Department of Computer Science\\
       University of Konstanz, Germany}

\maketitle

\begin{abstract}
We study offline reinforcement learning problems with a long-run average reward objective. The state-action pairs generated by any fixed behavioral policy thus follow a Markov chain, and the {\em empirical} state-action-next-state distribution satisfies a large deviations principle. We use the rate function of this large deviations principle to construct an uncertainty set for the unknown {\em true} state-action-next-state distribution. We also construct a distribution shift transformation that maps any distribution in this uncertainty set to a state-action-next-state distribution of the Markov chain generated by a fixed evaluation policy, which may differ from the unknown behavioral policy. We prove that the worst-case average reward of the evaluation policy with respect to all distributions in the shifted uncertainty set provides, in a rigorous statistical sense, the least conservative estimator for the average reward under the unknown true distribution. This guarantee is available even if one has only access to one single trajectory of serially correlated state-action pairs. The emerging robust optimization problem can be viewed as a robust Markov decision process with a non-rectangular uncertainty set. We adapt an efficient policy gradient algorithm to solve this problem. Numerical experiments show that our methods compare favorably against state-of-the-art methods.
\end{abstract}


\section{Introduction} 

Recent advances in reinforcement learning have led to remarkable performance improvements in sequential decision-making across various domains, including strategic gameplay \citep{ref:Silver_2016, openai2019solving}, robotic control \citep{ref:Andrychowicz-20}, autonomous teaching \citep{mandel2014offline}, or online recommendation \citep{liu2018deep} among others. Reinforcement learning is particularly successful whenever online data can be acquired through repeated, low-cost interactions with the environment \citep{sutton2018reinforcement, bertsekas2023course}. In many applications, however, continuous and/or inexpensive interaction with the system is not feasible, limiting the applicability of traditional reinforcement learning methods. In such cases, one must rely on {\em offline reinforcement learning}, which learns an optimal policy from pre-collected data without the opportunity for further exploration \citep{ref:levine2020offline}. Offline reinforcement learning is attractive when active experimentation is prohibitively costly or unethical. It is widely used, for instance, in education \citep{mandel2014offline}, healthcare \citep{oberst2019counterfactual} or marketing~\citep{gottesman2019guidelines}.

Throughout this paper we study offline reinforcement learning problems that seek a stationary policy with maximum average reward for an infinite-horizon tabular Markov decision process (MDP). While the transition kernel of the MDP is assumed to be unknown, the agent has access to a single finite state-action trajectory generated under a fixed behavioral policy, which may or may not be known. In offline reinforcement learning, it is common to distinguish {\em off-policy evaluation} and {\em offline policy optimization} problems. Off-policy evaluation seeks an accurate estimate for the average reward of a fixed evaluation policy (which typically differs from the behavioral policy that generates the data). Offline policy optimization, on the other hand, uses an off-policy evaluation oracle in order to find a policy that maximizes long-run average reward.

The effectiveness of an offline reinforcement learning algorithm heavily relies on the reliability of the underlying off-policy evaluation oracle. To our best knowledge, the vast majority of commonly used oracles are designed for the discounted reward criterion, with few exceptions such as~\citep{zhang2021average,saxena2023off}. In addition, they typically require access to multiple independent state-action trajectories generated under the same behavioral policy. Arguably the simplest approach to off-policy evaluation is the {\em direct method}, which uses the available data trajectories to estimate a parametric model for the discounted reward of the evaluation policy. For example, the discounted reward is representable as a function of the transition kernel or the action-value function, which can be inferred via maximum likelihood estimation \citep{ref:Mannor-07} or linear regression \citep{ernst2005tree,le2019batch}, respectively, 
see also \citep{ref:Csaba:OPE-06, ref:Lagoudakis-03}. However, the resulting reward estimators are prone to significant bias for several reasons. First, the discounted reward depends nonlinearly on the transition kernel. Thus, unbiased estimators for the transition kernel result in biased reward estimators. This bias is most pronounced if the likelihood ratio between the evaluation policy and the behavioral policy is high, that is, if some actions are chosen much more often under the evaluation policy than under the behavioral policy. Second, parametric estimators for the action-value function are typically biased due to model misspecification. Importance sampling methods reweight the observed rewards using the likelihood ratio between the evaluation and behavioral policies \citep{ref:Precup-00, ref:Hirano-03, ref:Joachims-15}. This approach obviates an explicit parameterization of the value function but suffers from high variance, which grows rapidly with the length of the observed state-action trajectories. To mitigate these limitations, doubly robust methods combine direct estimation with an importance sampling correction~\citep{jiang2016doubly,thomas2016data}. These approaches use the direct method to construct a baseline estimator and correct residual errors with importance weighting. The resulting estimators remain unbiased if either the baseline estimator or the importance weights are accurate, while generally reducing variance compared to pure importance sampling and mitigating bias compared to the direct method---hence the attribute ``{\em doubly robust}.'' These methods are known to achieve the lowest asymptotic variance among unbiased estimators \citep{kallus2019efficiently}.

All off-policy evaluation methods discussed so far assume that the testing data is generated under the same transition kernel as the training data. This assumption is inappropriate if the transition kernel of the MDP may change over time because distribution shifts can significantly distort reward estimates \citep{wang2020reliable}. Distributionally robust off-policy evaluation methods compute the worst-case average reward of the evaluation policy over an ambiguity set of plausible transition kernels. This approach mitigates risks from model misspecification and sampling bias while yielding confidence bounds on the true average reward \citep{shi2022distributionally, ma2022distributionally, panaganti2022robust, bhardwaj2024adversarial, ramesh2024distributionally, ref:si-2020, kallus2022doubly}.  
Notably, state-of-the-art offline reinforcement learning methods often build on pessimistic estimators that underestimate the average reward \citep{yin2021towards, yan2023efficacy, ref:Uehara-23, ref:Hu-25} or work with regularized value functions \citep{xie2021bellman,zhan2022offline, fakoor2021continuous}.


Recall that offline policy optimization estimates the highest achievable reward by maximizing the reward estimate provided by an off-policy evaluation oracle over all admissible policies. It is well known that—for any fixed training sample size—this estimated maximum is {\em optimistically} biased, even when the underlying off-policy evaluation oracle is unbiased. This is simply a manifestation of the notorious optimizer’s curse \citep{smith2006optimizer}. In finite samples, the optimistic bias in the estimated maximum can only be reduced (let alone eliminated) by employing a {\em pessimistically} biased off-policy evaluation oracle. Such oracles are designed to produce lower confidence bounds on the true maximum reward with a small significance level~$\beta\in(0,1)$. In data-driven optimization, $\beta$ is sometimes referred to as the {\em out-of-sample disappointment} \citep{van2021data}.

It is natural to call an off-policy evaluation oracle {\em efficient} if it is the least conservative among all oracles whose out-of-sample disappointment does not exceed a prescribed threshold. Thus, an efficient oracle strikes an optimal trade-off between the reward it predicts (which should be as {\em high} as possible) and the out-of-sample disappointment it incurs (which should be as {\em low} as possible). We aim to construct an off-policy evaluation oracle for tabular MDPs that is efficient in this sense. On a high level, our construction can be explained as follows. We first show that there is a one-to-one correspondence between the stationary state-action-next-state distribution of a controlled MDP and the combination of the MDP's transition kernel and the applied control policy. Thus, the state-action-next-state distribution corresponding to the behavioral policy encapsulates all the information needed to compute the long-run average reward of any given evaluation policy. Next, we demonstrate that the \emph{empirical} state-action-next-state distribution, derived from the observed state-action trajectory of the behavioral policy, serves as a consistent estimator for the true state-action-next-state distribution. Moreover, this estimator satisfies a large deviations principle with a rate function reminiscent of the conditional relative entropy. Finally, we estimate the average reward of the evaluation policy under the unknown true transition kernel by considering the worst-case (infimal) average reward across all transition kernels whose state-action-next-state distributions under the behavioral policy deviate from the empirical state-action-next-state distribution by at most~$\rho$, as measured by the rate function of the large deviations principle. We prove that the out-of-sample disappointment of the resulting distributionally robust off-policy evaluation oracle decays exponentially at a rate of~$\rho$ with the length of the observation history. Furthermore, we show that this oracle is the least conservative among all oracles whose out-of-sample disappointment decays at least at rate~$\rho$, making the proposed off-policy oracle asymptotically efficient.

We emphasize that our efficiency guarantees remain valid even when only a single trajectory of serially correlated states and actions is available, and even if the behavioral policy generating this trajectory is unknown. Moreover, we prove that our efficient off-policy evaluation oracle naturally leads to an efficient estimate for the optimal value of the corresponding offline reinforcement learning problem. Computing this optimal value requires solving a robust MDP with a non-rectangular uncertainty set. Unfortunately, such problems are generically NP-hard \citep{wiesemann2013robust}. To address this, we tailor the actor-critic algorithm proposed by \citet{li2023policy}, originally designed for robust discounted reward MDPs with arbitrary non-rectangular uncertainty sets, to the problem at hand. Given an oracle for approximating the (NP-hard) robust policy evaluation subproblem, we show that this algorithm finds an $\epsilon$-optimal solution for the corresponding policy improvement problem in $O(1/\epsilon^4)$ iterations. Approximate solutions for the robust policy evaluation subproblem can be obtained using the randomized projected Langevin dynamics algorithm proposed by \citet{li2023policy}. Although the runtime of this algorithm scales exponentially with the number of states and actions, numerical experiments suggest that it remains effective in practice.

In a nutshell, the main contributions of this paper can be summarized as follows.
\begin{itemize}  
\item We propose a novel approach to offline reinforcement learning that applies even when only a \textit{single trajectory} of correlated data is available, generated under an \textit{unknown} behavioral policy.  

\item We develop a distributionally robust off-policy evaluation oracle that is statistically efficient, thus optimally balancing in-sample performance against out-of-sample disappointment. Furthermore, we show that this oracle yields an efficient estimator for the optimal value of the corresponding offline reinforcement learning problem.  

\item Computing the proposed estimators reduces to solving a robust MDP with a non-rectangular uncertainty set. To address this hard problem, we adapt an existing actor-critic algorithm to solve the robust MDP approximately. Numerical experiments show that the resulting estimators are competitive with state-of-the-art baselines on standard test problems.  
\end{itemize}  

Our paper contributes to a stream of research that exploits large deviations theory in order to construct distributionally robust estimators for the optimal solutions of data-driven decision problems that enjoy statistical efficiency guarantees. \cite{van2021data} develop efficient estimators for {\em static} stochastic programs assuming only access to {\em independent} samples from the distribution of the uncertain problem parameters. They show that a distributionally robust optimization model with a relative entropy uncertainty set is statistically optimal. \citet{ref:Sutter-19} extend this model to more general data-generating processes with serially dependent observations. They show that, when the data process is governed by a parametric distribution and the underlying parameters admit a sufficient statistic satisfying a large deviations principle, then solving a distributionally robust optimization problem with an uncertainty set constructed via the rate function of the large deviations principle is statistically optimal. \cite{li2021distributionally} propose a customized Frank-Wolfe algorithm to compute efficient distributionally robust estimators under the assumption that the data is generated by a Markov chain. However, this algorithm only guarantees convergence to a stationary point, which may lack the efficiency properties of the global optimizer. \cite{bennouna2021learning} explore the same setting as \cite{van2021data}, analyzing the effects of imposing different decay rates on out-of-sample disappointment. They show that distributionally robust estimators with a relative entropy uncertainty sets are optimal in the exponential regime, variance-regularized empirical estimators are optimal in the sub-exponential regime, and worst-case robust estimators are optimal in the super-exponential regime. 
Similarly, \citet{ref:ganguly-2024} examine the construction of confidence intervals in the moderate deviation regime and establish the efficiency of distributionally robust estimators.
Compared to all these works, our model is the only one to simultaneously offer the following benefits. It realistically assumes that the available data is limited to a single finite trajectory of states and actions generated by an MDP controlled by a stationary behavioral policy. It does {\em not} rely on the restrictive assumption that all transitions between arbitrary state-action pairs have a positive probability. It recognizes that transitions between certain state-action pairs must have identical probabilities because the behavioral policy is stationary. It explicitly accounts for the distribution shift between the state-action trajectory distributions under the behavioral policy and the evaluation policy, respectively. In addition, our model delivers estimators that enjoy both asymptotic and finite-sample guarantees on the out-of-sample disappointment. Finally, we provide an algorithm that solves the relevant robust MDPs to global optimality, thus ensuring that the proposed efficient estimators are accessible. \cite{sutter2021robust} also study off-policy evaluation problems that explicitly account for distribution shifts. However, they assume to have access to independent samples from the stationary state-action-next-state distribution, and they assume that only the reward function is unknown, whereas the transition kernel and the behavioral policy are known. Finally, their methods do not extend easily to offline policy optimization. 


\paragraph{Structure.}
The rest of the paper is structured as follows. Section~\ref{sec:OPE} reviews and extends the large deviations theory for Markov chains and Markov decision processes. Leveraging this theory, Sections~\ref{ssec:OPE:dist:shift} and~\ref{ssec:OPL:dist:shift} propose statistically optimal solutions for the off-policy evaluation and the offline policy optimization problems, respectively. Section~\ref{sec:computational:aspects} then develops a projected Langevin dynamics method for solving the robust policy evaluation problem and an actor-critic method for solving the offline policy optimization problem. Finally, Section~\ref{sec:numerical:experiments} validates the efficiency of the proposed estimators on standard test problems from reinforcement learning and operations research.

\paragraph{Notation.}
The probability simplex over a finite set $\mc X$ is defined as $\Delta(\mc X) = \{p \in \mb R^{|\mc X|}_+ : \sum_{x\in\mc X} p(x) = 1\}$. The relative entropy of $p\in\Delta(\mc X)$ with respect to $q\in\Delta(\mc X)$ is defined as $\mathsf{D}(p \| q)=\sum_{x\in\mc X} p(x) \log \left(p(x)/q(x)\right)$, where we use the conventions $0\log(0/t)=0$ for any~$t\geq 0$ and $t\log(t/0)=\infty$ for any~$t>0$. The support of~$p\in\Delta(\mc X)$ is the set $\supp (p) = \{x\in\mc X: p(x)>0\}.$
\section{Statistics of Markov Chains and Markov Decision Processes} \label{sec:OPE}

Off-policy evaluation and offline policy optimization constitute statistical learning problems based on Markovian data. In Section~\ref{sec:PMDI} we thus review and generalize a large deviations principle for the doublet distribution of an irreducible Markov chain, which provides a mathematical framework for studying the probabilities of rare events. Building on these insights, in Section~\ref{ssec:problem:statement} we then derive a large deviations principle for the state-action-next-state distribution of a Markov decision process.

\subsection{Markov Chains}\label{sec:PMDI}

Consider
a time-homogeneous irreducible Markov chain given by a triple $(\mc X, P, \gamma)$ consisting of a finite state space $\mc X=\{1,\ldots, d\}$, a transition probability matrix $P\in \Delta(\mc X)^d$ and an initial distribution $\gamma\in \Delta(\mc X)$. If the system underlying the Markov chain is in state $x_t\in\mc X$ at time~$t$, then it moves to state~$x_{t+1}\in\mc X$ at time~$t+1$ with probability $P(x_t,x_{t+1})$. Thus, $P(x_t,\cdot)$ represents the distribution of~$X_{t+1}$ conditional on $X_t=x_t$. There exists a unique probability measure $\mb P_P$ defined on the canonical sample space $\Omega=\mc X^\infty$ equipped with its power set $\sigma$-algebra $\mc F=2^\Omega$ such that
\[
    \mb P_P(X_1=x_1)=\gamma(x_1)\quad \forall x_1\in\mc X
\]
and
\[
    \mb P_P(X_{t+1}=x_{t+1}| X_t=x_t,\ldots ,X_1=x_1)=P(x_t,x_{t+1})\quad \forall x_1,\ldots x_{t+1}\in\mc X.
\]
Note that $\mb P_P$ also depends on~$\gamma$, but we suppress this dependence to avoid clutter. As the Markov chain at hand is irreducible, the Perron-Frobenius theorem implies that there exists a unique stationary state distribution $\mu \in \Delta(\mathcal{X})$ that satisfies $\mu(y)=\sum_{x\in\mc X}\mu(x) P(x,y)>0$ for all $y\in\mc X$. Using~$P$ and $\mu$, we can further define the stationary doublet distribution $\theta \in\Delta(\mc X\times\mc X)$ through $\theta(x,y)=\mu(x)P(x,y)$ for all $x,y\in\mc X$. From the last formula it is evident that the transition probability matrix~$P$ can be recovered from~$\theta$, that is, we have $P(x,y)=\theta(x,y)/\mu(x)$ for all~$x,y\in\mc X$. Hence, there is a one-to-one correspondence between~$P$ and~$\theta$. Without loss of generality, we can thus denote the probability measure governing the Markov chain $\{X_t\}_{t=1}^\infty$ by~$\mb P_\theta$ instead of~$\mb P_P$. We prefer~$\theta$ to~$P$ because it admits a simple estimator that satisfies a large deviations principle. 

Note that $\theta$ has balanced marginals in the sense that
\begin{equation}
    \label{eq:balanced-marginals}
    \sum_{x\in\mc X} \theta(x,y)= \mu(y) = \sum_{y\in\mc X} \theta(y,x) \quad\forall y\in\mc X,
\end{equation}
where the first equality follows from the definition of~$\mu$, and the second equality follows from the definition of~$\theta$. In the following we use~$\Theta$ to denote the set of all doublet distributions~$\theta\in\Delta(\mc X\times\mc X)$ that satisfy~\eqref{eq:balanced-marginals}. We emphasize that not every~$\theta\in\Theta$ represents the doublet distribution of an irreducible Markov chain, that is, it is not sufficient for~$\theta$ to have balanced marginals. In addition, the pair $(\mc X,\supp (\theta))$ must represent a strongly connected directed graph with vertex set~$\mc X$ and edge set~$\supp(\theta)$. The strong connectedness ensures that the underlying Markov chain has only one single communicating class of states. In the following, we use $\Theta_0\subseteq \Theta$ to denote the set of all doublet distributions that induce an irreducible Markov chain.

The Markov law of large numbers \citep[Theorem~1.7.6]{norris1998markov} implies that
\begin{equation}\label{eq:def:doublet:frequency}
    \theta(x,y)=\lim_{T\to\infty}\frac{1}{T}\sum_{t=1}^T\mb P_\theta (X_t=x, X_{t+1}=y) \quad \forall x,y\in\mc X, \quad\forall\theta\in\Theta_0.
\end{equation}
Given a finite trajectory of state observations $\{ X_t\}_{t=1}^T$, a natural estimator for~$\theta$ is thus given by the empirical doublet distribution $\widehat\theta_T\in\Delta(\mc X\times\mc X)$, which is defined through
\begin{equation} \label{MC:estimator}
    \widehat{\theta}_{T}(x,y)=\frac{1}{T} \left( \sum_{t=1}^{T-1} \indic{{X}_{t}=x,{X}_{t+1}=y} + \indic{{X}_{T}=x,{X}_1=y}\right)\quad \forall x,y\in\mc X.
\end{equation}
The ghost transition from $X_T$ to $X_1$ \citep{vidyasagar2014elementary} in this definition ensures that~$\thetahat$ has balanced marginals and thus lies in~$\Theta$. In the following, we refer to elements of~$\Theta_0$ as models and elements of~$\Theta$ as estimator realizations.
It is natural to measure the discrepancy between a model $\theta\in\Theta_0$ and an estimator realization $\theta'\in\Theta$ by their conditional relative entropy. 
\begin{definition}[Conditional relative entropy for Markov chains] 
\label{def:conditional_relative_entropy}
The conditional relative entropy of $\theta'\in\Theta$ with respect to $\theta\in \Theta_0$ is defined as
\begin{align}\label{def:Dc:Theta}
	\mathsf{D_{mc}}(\theta'\|\theta)
	=\sum_{x,y\in\mc X} \theta'(x,y) \left( \log \frac{\theta'(x,y)}{\sum_{z\in\mc X} \theta'(x,z)}  -  \log \frac{\theta(x,y)}{\sum_{z\in\mc X} \theta(x,z)}\right).
\end{align}
\end{definition}
The conditional relative entropy is well-defined for all $\theta'\in \Theta$ and $\theta\in\Theta_0$. Indeed, due to our conventions for the logarithm, $\mathsf{D_{mc}}(\theta'\|\theta)$ is finite whenever $\supp(\theta')\subseteq\supp(\theta)$ and evaluates to~$+\infty$ otherwise. If the doublet distributions~$\theta'$ and~$\theta$ belong to~$\Theta_0$, then they induce irreducible Markov chains with unique transition probability matrices~$P'$ and~$P$ and stationary distributions~$\mu'$ and~$\mu$, respectively. In this case, the conditional relative entropy can be equivalently expressed as
\begin{align}\label{Dc:weighted:transition:lratio}
   \mathsf{D_{mc}}(\theta'\|\theta)= \sum_{x,y\in\mc X} \theta'(x,y) \log \frac{P'(x,y)}{P(x,y)}
=\sum_{x\in\mc X} \mu'(x) \, \mathsf{D}(P'(x,\cdot)\|P(x,\cdot)).
\end{align}
The last formula in~\eqref{Dc:weighted:transition:lratio} motivates the name ``conditional relative entropy.'' In addition, $\Dc{\theta'}{\theta}$ admits a unique lower semi-continuous extension to $\Theta\times\Theta$, which is obtained by setting 
$$
    \Dc{\theta'}{\theta}=\lim_{\delta \downarrow 0}\inf_{(\vartheta',\vartheta)\in\Theta\times\Theta_0} \left\{ \Dc{\vartheta'}{\vartheta}:\|(\vartheta',\vartheta)-(\theta',\theta)\|\le\delta \right\} \quad\forall (\theta',\theta)\in\Theta\times(\Theta\backslash\Theta_0);
$$
see \citep[p.~1996]{ref:Sutter-19} and \citep[Definition 1.5]{rockafellar2009variational}. In the following, we will always mean this lower semi-continuous extension to $\Theta\times\Theta$ when referring to~$\Dc{\theta'}{\theta}$. 

The conditional relative entropy is significant because it represents the rate function of a large deviations principle for the empirical doublet distribution $\widehat \theta_T$. Large deviations theory provides bounds on the exponential rate at which the probability of a rare event $\widehat\theta_T\in\mc D$ decays as the length $T$ of the observation history grows. These bounds are expressed in terms of the infimum of a rate function over the set~$\mc D$ or its interior.
The classical large deviations principle for Markov chains~{\cite[Theorem~1]{natarajan1985large}} assumes that $\supp(\theta)= \mc X\times \mc X$. We generalize this classical result to arbitrary irreducible Markov chains. This generalization necessitates only cosmetic changes in the proof. We provide a proof sketch to keep this paper self-contained.

\begin{theorem}[Large deviations principle for Markov chains] 
\label{thm:LDP}
For all $\theta \in \Theta_0$ and Borel sets $\mathcal{D} \subseteq \Theta$, the empirical doublet distribution $\widehat{\theta}_{T}$ defined in~\eqref{MC:estimator} satisfies
$$
\begin{aligned}
    -\inf _{\theta^{\prime} \in \operatorname{int} \mathcal{D}} \mathsf {D_{mc}}(\theta^{\prime}\| \theta) & \leq \liminf _{T \rightarrow \infty} \frac{1}{T} \log \mathbb{P}_{\theta}\left(\widehat{\theta}_{T} \in \mathcal{D}\right) \\
    & \leq \limsup_{T \rightarrow \infty} \frac{1}{T} \log \mathbb{P}_{\theta}\left(\widehat{\theta}_{T} \in \mathcal{D}\right) \leq-\inf _{\theta^{\prime} \in \mathcal{D}} \mathsf {D_{mc}}(\theta^{\prime}\| \theta).
\end{aligned}
$$
\end{theorem}
\begin{proof}
Fix an arbitrary $\theta\in\Theta_0$, and define $\Theta_T=\Theta\cap\frac{1}{T}\{0,\ldots,T\}^{d\times d}$
as the set of all possible realizations of~$\widehat\theta_T$. 
One readily verifies that
\begin{align*}
 &\frac{1}{T}\log \mb P_{\theta}\left(\thetahat\in\mc D\right)
 =\frac{1}{T}\log \mb P_{\theta}\left(\thetahat\in\mc D\cap \Theta_T\right)
 \le \frac{1}{T} \log\left(|\mc D\cap \Theta_T|\sup_{\theta'\in\mc D\cap \Theta_T} \mb P_\theta\left(\widehat\theta_T=\theta'\right)\right)
\\
&\qquad\leq  \frac{1}{T}\log \left(|\Theta_T| \sup_{\theta'\in\mc D}\mb P_{\theta}\left(\thetahat=\theta'\right)\right).
 \end{align*}
 Next, define the type class
 \begin{align}\label{expr:type:class}
    J(\theta')=\left\{\{x_t\}_{t=1}^{T}\in\mc X^T: \frac{1}{T} \left( \sum_{t=1}^{T-1} \indic{{x}_{t}=x,{x}_{t+1}=y} + \indic{{x}_{T}=x,{x}_1=y}\right)=\theta'(x,y) \ \forall x,y\in\mc X\right\}
 \end{align}
as the set of all sample paths consistent with the estimator realization $\theta'\in\Theta$. Thus, we have that $\widehat\theta_T=\theta'$ if and only if $\{X_t\}^{T}_{t=1}\in J(\theta').$ The above reasoning therefore implies that
\begin{align}
\frac{1}{T}\log  \mathbb{P}_{\theta}&\left(\widehat\theta_T\in\mc D\right)\nonumber\\
  &\leq \frac{d^{2}}{T}\log (T+1) +\frac{1}{T} \log\sup_{\theta'\in\mc D}\mb P_{\theta}\left(\thetahat=\theta'\right)\nonumber\\
    &\leq \frac{d^{2}}{T}\log (T+1) 
    +\frac{1}{T}  \sup_{\theta'\in\mc D}\log\left(|J(\theta')|\sup_{\{x_t\}_{t=1}^{T}\in J(\theta')}\mb P_{\theta}(X_t=x_t \ \forall t=1,\ldots,T)\right).\label{ineq:type:class:bound}
    \end{align}
Next, use $\mu$ and $\mu'$ to denote the stationary state distributions corresponding to the stationary doublet distributions $\theta$ and $\theta'$, respectively. It follows from~\citep[Theorem 11]{vidyasagar2014elementary} that
\begin{align}\label{type:card:bound}
-d^2\log(2T)+T\,\mathsf{H_c}(\theta')   \le \log |J(\theta')|\le \log T+T\, \mathsf{H_c}(\theta')    ,
\end{align}
where
\[\mathsf{H_c}(\theta')=\sum_{x\in\mc X}\mu' (x)\log \mu'(x)-\sum_{x,y\in\mc X}\theta'(x,y)\log \theta'(x,y)\]
is the conditional entropy of the estimator realization $\theta'\in\Theta.$
    Similarly, it follows from inequalities~(35) and (36) in~\citep{vidyasagar2014elementary} that
  \begin{equation}\label{type:prob:bound}
   -T \,(\mathsf{H_c}(\theta')+\Dc{{\theta'}}{\tb})+\underbar c  \le \log  \mb P_{\theta}(X_t=x_t \ \forall t=1,\ldots,T) \leq -T \,(\mathsf{H_c}(\theta')+\Dc{{\theta'}}{\tb})+\bar c,
  \end{equation}
  where $ \underbar c=\min_{x,y\in\mc X}\log(\mu(x)\mu(y))/\tb(x,y))$ and $ \bar c=\max_{x,y\in\mc X}\log(\mu(x)\mu (y))/\tb(x,y))$.
    Substituting these estimates into~\eqref{ineq:type:class:bound} yields
    \begin{align*}
            \frac{1}{T}\log  \mathbb{P}_{\theta}\left(\widehat\theta_T\in \mc D\right)
    &\leq\frac{d^{2}}{T}\log (T+1) + \frac{1}{T}\log (T)+\sup_{\theta'\in \mc D}(-\Dc{\theta'}{\theta}+\frac{1}{T}\bar c)\\
    &=\frac{d^{2}}{T}\log (T+1) + \frac{1}{T}(\log (T)+\bar c)-\inf_{\theta'\in \mc D}\Dc{\theta'}{\theta},
\end{align*}
from which the upper bound on $\limsup_{T \rightarrow \infty} \frac{1}{T} \log \mathbb{P}_{\theta}(\widehat{\theta}_{T} \in \mathcal{D})
$ follows immediately.

As for the lower bound, note that $\cup_{T\in\mb N} \Theta_T$ is dense in $\mathrm{int}\mc D$. 
Hence, there is $T_0\in\mb N$ that only depends on $\mc D$ and a deterministic sequence $\{\theta'_T\}_{T\in\mb N}$ such that $\theta'_T\in\Theta_T\cap\mathrm{int}\mc D$ for all $T\ge T_0$ and
\[\inf_{\theta'\in\mathrm{int} \mc D}\Dc{\theta'}{\theta}=\liminf_{T\to\infty}\Dc{\theta'_T}{\theta}.
\]
For all $T\ge T_0$,  
we then have
\begin{align*}
    \frac{1}{T}\log \mathbb{P}_{\theta}\left(\widehat\theta_T\in \mc D\right)
    &\ge \frac{1}{T}\log \mb P_\theta \left(\widehat\theta_T=\theta'_T\right)
    \\&\ge\frac{1}{T}\log \left(|J(\theta'_T)|\inf_{\{x_t\}_{t=1}^{T}\in J(\theta'_T)}\mb P_{\theta}(X_t=x_t \ \forall t=1,\ldots,T)\right)
    \\&\ge -\frac{d^2}{T}\log (2T)+\frac{1}{T}(\log T+\underbar c)-\Dc{\theta'_T}{\theta},
\end{align*}
where the first equality holds because $T\ge T_0$, and the third inequality follows again from~\eqref{type:card:bound} and~\eqref{type:prob:bound}. Taking the limit inferior as $T$ tends to infinity finally yields the desired lower bound.
\end{proof}

Corollary~\ref{ldp:finite} below establishes a finite-sample version of the large deviations upper bound in Theorem~\ref{thm:LDP}. Its proof follows immediately from that of Theorem~\ref{thm:LDP} and is therefore omitted. 
\begin{corollary}[Finite-sample version of Theorem~\ref{thm:LDP}]\label{ldp:finite}
\hspace{-0.2ex} For all $\theta \in \Theta_0$ and Borel sets $\mathcal{D}\subseteq \Theta$, the empirical doublet distribution $\widehat{\theta}_{T} $ defined in~\eqref{MC:estimator} satisfies
$$\begin{aligned}
\frac{1}{T}\log  \mathbb{P}_{\theta}\left(\widehat\theta_T\in \mc D\right)
\leq \frac{1}{T}(\log T+\bar c+ d^2\log (T+1) )-\inf_{\theta'\in \mc D}\Dc{\theta'}{\theta} \quad \quad\forall T \in \mb N,
\end{aligned}$$ 
where $\bar c>0$ is a universal constant that depends only on~$\theta.$
\end{corollary}

Besides representing the rate function of a large deviations principle, the conditional relative entropy~$\mathsf{D}_{\mathsf{mc}}$ has many useful properties including level compactness, radial monotonicity and coercivity. We refer to~\citep[Proposition~5.1]{ref:Sutter-19} for a discussion of these properties.

\subsection{Markov Decision Processes}\label{ssec:problem:statement}
Consider a \textit{Markov decision process} (MDP) given by a five-tuple $(\mathcal{S}, \mathcal{A}, Q, r, \eta)$ consisting of a finite state space $\mathcal{S}=\{1,\ldots,S\}$, a finite action space $\mathcal{A}=\{1,\ldots,A\}$, a transition kernel $Q\in   \mc Q=\Delta(\mathcal{S})^{SA}$, a reward-per-stage function $r: \mathcal{S} \times \mathcal{A} \rightarrow \mathbb{R}$, and an initial distribution $\eta\in\Delta(\mathcal{S})$. If the system underlying the MDP is in state $s_t\in\mc S$ at time $t$ and action $a_t\in\mc A$ is executed, then an immediate reward $r(s_t,a_t)$ is earned, and the system moves to state $s_{t+1}$ at time $t+1$ with probability $Q(s_{t+1}|s_t,a_t)$. Thus, $Q(\cdot|s_t,a_t)$ represents the distribution of $S_{t+1}$ conditional on $S_t=s_t$ and $A_t=a_t$. Actions are chosen according to a stationary policy, which is described by a stochastic kernel $\pi\in\Pi=\Delta(\mc A)^S$. That is, the probability of choosing action $a_{t}$ if the current state is~$s_t$ is characterized by $\pi(a_{t}|s_{t})$. Thus, $\pi(\cdot|s_{t})\in \Delta(\mc A)$ represents the distribution of~$A_{t}$ conditional on~$S_{t}=s_{t}$. Under a stationary policy $\pi$, there exists a unique probability measure $\mathbb{P}_{\pi,Q}$ defined on the canonical sample space $\Omega=(\mathcal{S}\times\mathcal{A})^\infty$ equipped with its power set $\sigma$-algebra $\mc F=2^\Omega$
such that 
\begin{subequations}
    \begin{align}
    \label{expr:initial:distribution:pi}
          \mathbb{P}_{\pi,Q}(S_1=s_1)=\eta(s_1)\quad\forall s_1\in\mc S,
   \end{align}
   and for all $t\in\mb N$ we have
   \begin{equation}
   \label{def:Q:pi}
   \begin{aligned}
    &\mathbb{P}_{\pi,Q}(S_{t+1}=s_{t+1}|S_t = s_t, A_t=a_t,\ldots,S_1=s_1,A_1=a_1) \\
    &\hspace{3cm}= Q(s_{t+1}|s_t,a_t) \quad\forall s_1,\ldots,s_{t+1}\in\mc S,\;a_1,\ldots,a_t\in\mc A,
    \end{aligned}
    \end{equation}
    and
    \begin{equation}
    \label{def:pi:pi}
    \begin{aligned}
        &\mathbb{P}_{\pi,Q}(A_{t+1}=a_{t+1}|S_{t+1}=s_{t+1},S_t = s_t, A_{t}=a_{t},\ldots,S_1=s_1,A_1=a_1) \\
     &\hspace{3cm}= \pi(a_{t+1}|s_{t+1})\quad\forall s_1,\ldots,s_{t+1}\in\mc S,\;a_1,\ldots,a_{t+1}\in\mc A.
\end{aligned}
\end{equation} 
\end{subequations}
Further details about the construction of $\mb P_{\pi,Q}$ are provided in~\citep[\S~2.2]{hernandez2012discrete}. Note that $\mb P_{\pi,Q}$ also depends on~$\eta$, but we suppress this dependence notationally to avoid clutter. To simplify notation, we will use $X_t$ as a shorthand for the state-action pair~$(S_t,A_t)$. Note that $X_t$ ranges over $\mc X = \mathcal{S}\times\mathcal{A}$. As in Section~\ref{sec:PMDI}, $d=SA$ denotes the cardinality of~$\mc X$. 

\begin{proposition}[State-action process $\{X_t\}_{t=1}^\infty$]\label{prop:Xt:MC:ergo}
The stochastic process $\{X_t\}_{t=1}^\infty$ represents a time-ho\-mogeneous Markov chain under $\mb P_{\pi,Q}$, and its transition probability matrix $P$ satisfies
\begin{align}\label{expr:P:from:pi:Q}
    P((s,a),(s',a'))=\pi(a'|s')Q(s'|s,a)\quad\forall s,s'\in\mc S,\;a,a'\in\mc A.
\end{align}
\end{proposition}
\begin{proof}
By the construction of $\mb P_{\pi,Q},$ 
we have for all $t\in\mb N$ that
\begin{align*}
&\mb P_{\pi,Q}(X_{t+1}=x_{t+1}|X_t=x_t,X_{t-1}=x_{t-1},
    \ldots,X_1=x_1)\\
    &\qquad=\mb P_{\pi,Q}(S_{t+1}=s_{t+1},A_{t+1}=a_{t+1}|S_t=s_t,A_t=a_t,\ldots,S_1=s_1,A_1=a_1)
    \\&\qquad=\mb P_{\pi,Q}(A_{t+1}=a_{t+1}|S_{t+1}=s_{t+1},S_t=s_t,A_t=a_t,\ldots,S_1=s_1,A_1=a_1)\\&\quad\qquad\;\times\mb P_{\pi,Q}(S_{t+1}=s_{t+1}|S_t=s_t,A_t=a_t,\ldots,S_1=s_1,A_1=a_1)
    \\&\qquad=\pi(a_{t+1}|s_{t+1})Q(s_{t+1}|s_t,a_t)\quad\forall s_1,\ldots,s_{t+1}\in\mc S,\; a_1,\ldots,a_{t+1}\in\mc A,
\end{align*}
where the second equality follows from Bayes' law, while the third equality exploits~\eqref{def:Q:pi} and~\eqref{def:pi:pi}. Thus, the stochastic process $\{X_t\}_{t=1}^\infty$ satisfies the Markov property, and its transition probability matrix satisfies~\eqref{expr:P:from:pi:Q}. Since $\pi$ and $Q$ are independent of $t,$ the Markov chain is time-homogeneous.
\end{proof}

Proposition~\ref{prop:Xt:MC:ergo} implies that the triple $(\mc X, P,\gamma)$ induces a Markov chain with initial distribution $\gamma\in\Delta(\mc X)$ defined through $\gamma(s,a)=\eta(s)\pi(a,s)$ for all~$s\in\mc S$ and~$a\in\mc A$. For this Markov chain to be irreducible, we require that~$\pi>0$. The family of all stationary policies with this property is denoted by~$\Pi_0$. In addition, we require that the directed graph $(\mc X,\mc E)$ with edge set
\[
    \mc E= \left\{((s,a),(s',a'))\in(\mc S\times\mc A)^2 \,:\, Q(s'|s,a)>0 \right\}
\]
is strongly connected. The family of all transition kernels with this property is denoted by~$\mc Q_0$. The following lemma is elementary, and therefore its proof is omitted.



\begin{lemma}\label{lem:irreducibleity:pi:Q:X_t}
If $\pi\in\Pi_0$ and $Q\in\mc Q_0,$ then the Markov chain~$\{X_t\}_{t=1}^\infty$ is irreducible under~$\mb P_{\pi,Q}$. Conversely, if $\pi\in\Pi\backslash\Pi_0$ or $Q\in\mc Q\backslash\mc Q_0,$ then the Markov chain~$\{X_t\}_{t=1}^\infty$ is not irreducible under~$\mb P_{\pi,Q}$.
\end{lemma}

As the Markov chain induced by~$\pi\in\Pi_0$ and~$Q\in\mc Q_0$ is irreducible, we know from Section~\ref{sec:PMDI} that it admits a unique stationary distribution~$\mu\in\Delta(\mc S\times\mc A)$ as well as a unique stationary doublet distribution $\theta\in\Delta((\mc S\times\mc A)^2)$. In addition, $\pi$ and~$Q$ induce a unique stationary state-action-next-state distribution $\xi\in\Delta(\mc S\times\mc A \times \mc S)$ defined through $\xi(s,a,s')=\sum_{a'\in\mc A} \theta((s,a),(s',a'))$ for all $s,s'\in\mc S$ and~$a\in \mc A$. Recall from Section~\ref{sec:PMDI} that~$\Theta_0$ denotes the family of doublet distributions corresponding to irreducible Markov chains on~$\mc X=\mc S\times\mc A$. We emphasize, however, that not every Markov chain on~$\mc X$ with doublet distribution~$\theta\in\Theta_0$ has a transition probability matrix of the form~\eqref{expr:P:from:pi:Q}. Instead, the state-action-next-state distribution~$\xi$, which has fewer degrees of freedom than~$\theta$, provides sufficient information to reconstruct~$\pi$ and~$Q$. To see this, note first that
\begin{equation}\label{xi:alt:pi:Q}
    \xi(s,a,s')=Q(s'|s,a)\mu(s,a) \quad \forall s,s'\in\mc S,\; a,a'\in\mc A
\end{equation}
and that $\mu(s,a) = \sum_{s'\in \mc S} \xi(s,a,s')$ for all $s\in\mc S$ and~$a\in\mc A$. Hence, we have
\begin{subequations}
\label{eq:q-pi-from-xi}
\begin{equation}\label{eq:q-from-xi}
        Q(s'|s,a) = \frac{\xi(s,a,s')}{\mu(s,a)} = \frac{\xi(s,a,s')}{\sum_{\tilde s\in\mc S}\xi(s,a,\tilde s)}\quad \forall s,s'\in\mc S,\; a\in\mc A.
\end{equation}
In addition, the defining equation of the stationary distribution~$\mu$ implies that
\[
    \mu(s',a')= \sum_{s\in\mc S} \sum_{a\in\mc A} \mu(s,a) P((s,a), (s',a')) = \sum_{s\in\mc S} \sum_{a\in\mc A} \mu(s,a) \pi(a'|s')Q(s'|s,a)\quad \forall s'\in\mc S,\; a'\in\mc A,
\]
where the second equality follows from Proposition~\ref{prop:Xt:MC:ergo}. This readily implies that
\begin{equation}\label{eq:pi-from-xi}
        \pi(a'|s') = \frac{\sum_{s\in\mc S} \sum_{a\in\mc A} \mu(s,a)Q(s'|s,a)}{\mu(s',a')} = \frac{\sum_{s\in\mc S} \sum_{a\in\mc A} \xi(s,a,s')}{\sum_{s\in\mc S} \xi(s,a',s')} \quad \forall s'\in\mc S,\; a'\in\mc A,
\end{equation}
\end{subequations}
where the second equality exploits~\eqref{xi:alt:pi:Q}. In summary, \eqref{eq:q-pi-from-xi} shows that both~$\pi$ and~$Q$ can indeed be reconstructed from~$\xi$. Hence, there is a one-to-one correspondence between~$(\pi,Q)$ and~$\xi$. Without loss of generality, we can thus denote the probability measure governing the Markov chain $\{X_t\}_{t=1}^\infty$ by~$\mb P_\xi$ instead of~$\mb P_{\pi,Q}$. We prefer~$\xi$ to~$(\pi,Q)$ because it admits a simple estimator akin to~$\widehat\theta_T$. 

Note that $\xi$ has balanced marginals in the sense that
\begin{equation}
    \label{eq:balanced-marginals-xi}
    \sum_{s' \in \mc S} \sum_{ a\in\mc A}\xi(s, a, s') =  \sum_{s' \in \mc S} \sum_{ a\in\mc A}\xi(s', a, s)\quad \forall s\in\mc S.
\end{equation}
In the following we use~$\Xi$ to denote the set of all~$\xi\in\Delta(\mc S\times\mc A\times\mc S)$ that satisfy~\eqref{eq:balanced-marginals-xi}. We emphasize that not every~$\xi\in\Xi$ is induced by an irreducible Markov chain. To see this, recall that if~$\pi\in\Pi_0$ and~$Q\in\mc Q_0$, then the Markov chain $\{X_t\}_{t=1}^\infty$ is irreducible, and thus~$\mu>0$. By~\eqref{xi:alt:pi:Q}, this ensures that the edge set~$\mc E$ admits the equivalent representation
\[
    \mc E= \left\{((s,a),(s',a'))\in(\mc S\times\mc A)^2 \,:\, \xi(s,a,s')>0 \right\}.
\]
Hence, if~$\pi\in\Pi_0$ and~$Q\in\mc Q_0$, then $(\mc X,\mc E)$ must represent a strongly connected graph. We use $\Xi_0\subseteq \Xi$ to denote the set of all $\xi\in\Xi$ with this property. By what has been said above, it is now easy to prove that every $\pi\in\Pi_0$ and~$Q\in\mc Q_0$ induces a unique~$\xi\in\Xi_0$ and vice versa.

Given a finite state-action trajectory $\{ (S_t,A_t)\}_{t=1}^T$, a natural estimator for~$\xi$ is thus the empirical state-action-next-state distribution $\qhat\in\Delta(\mc S\times\mc A\times\mc S)$ defined through
\begin{equation}
    \label{MDP:estimator}
    \qhat(s,a,s')=\frac{1}{T} \left( \sum_{t=1}^{T-1} \indic{{S}_{t}=s,A_t=a,{S}_{t+1}=s'} + \indic{{S}_{T}=s,A_T=a,{S}_1=s'}\right)\quad \forall s,s'\in\mc S,\;a\in\mc A.
\end{equation}
The ghost transition from $(S_T, A_T)$ to~$S_1$ in the above definition ensures that the estimator $\qhat$ has balanced marginals and thus lies in~$\Xi$. One readily realizes that $\qhat=F(\thetahat)$, where~$\thetahat$ denotes the empirical doublet distribution defined in~\eqref{MC:estimator}, and $F:\Theta\to \Xi$ satisfies
\begin{equation}
    \label{eq:F-definition}
    F(\theta)(s,a,s')=\sum_{a'\in\mc A} \theta((s,a), (s',a'))\quad \forall s,s'\in\mc S,\; a\in\mc A.   
\end{equation}

We henceforth refer to elements of~$\Xi_0$ as models and elements of~$\Xi$ as estimator realizations. The discrepancy between a model $\xi\in\Xi_0$ and an estimator realization $\xi'\in\Xi$ is naturally measured by a conditional relative entropy tailored to MDPs.

\begin{definition}[Conditional relative entropy for MDPs] 
\label{def:Lc}
The conditional relative entropy of $\xi'\in\Xi$ with respect to $\xi\in \Xi_0$ is defined as
\begin{equation}\label{def:Dm:Theta}
\begin{split}
	\Lc{\xi'}{\xi}&=\sum_{s,s'\in\mc S, a\in\mc A}\xi'(s,a,s')\left(\log \frac{\xi'(s,a,s')}{\sum_{\tilde s\in\mc S,\tilde a\in\mc A}\xi'(s,\tilde a,\tilde s)}-\log \frac{\xi(s,a,s')}{\sum_{\tilde s\in\mc S,\tilde a\in\mc A}\xi(s,\tilde a,\tilde s)}\right).
\end{split}
\end{equation}
\end{definition}
Due to our conventions for the logarithm, $\Lc{\xi'}{\xi}$ is finite whenever $\supp(\xi')\subseteq\supp(\xi)$ and evaluates to $+\infty$ otherwise. In addition, $\Lc{\xi'}{\xi}$ admits a unique lower semi-continuous extension to $\Xi\times\Xi$, which is obtained by setting 
$$
    \Lc{\xi'}{\xi}=\lim_{\delta \downarrow 0}\inf_{(\zeta',\zeta)\in\Xi\times\Xi_0} \left\{\Lc{\zeta'}{\zeta}:\|(\zeta',\zeta)-(\xi',\xi)\|\le\delta \right\}\quad\forall (\xi',\xi)\in\Xi\times(\Xi\backslash\Xi_0);
$$
see~\citep[p.~1996]{ref:Sutter-19} and \citep[Definition 1.5]{rockafellar2009variational}. In the following, we will always mean this lower semi-continuous extension to $\Xi\times\Xi$ when referring to~$\Lc{\xi'}{\xi}$. The next proposition shows that the conditional relative entropy~$\mathsf{D}_{\mathsf{mdp}}$ for MDPs is closely related to the conditional relative entropy~$\mathsf{D}_{\mathsf{mc}}$ for Markov chains introduced in Definition~\ref{def:conditional_relative_entropy}.

\begin{proposition}[Relation between $\mathsf{D}_{\mathsf{mc}}$ and $\mathsf{D}_{\mathsf{mdp}}$]
\label{prop:Dmdp-Dmc-relation}
If~$G: \Xi\to \Theta$ is defined through
\begin{equation}
    G(\xi) ((s,a),(s',a')) = \left\{\begin{array}{ll}
            \frac{\sum_{\tilde s\in\mc S}\xi(s',a',\tilde s)}{\sum_{\tilde s\in\mc S,\tilde a\in \mc A}\xi(s',\tilde a,\tilde s)} \, \xi(s,a,s') & \text{if } \sum_{\tilde s\in\mc S}\xi(s',a',\tilde s)>0
            \\
            0 & \text{if } \sum_{\tilde s\in\mc S}\xi(s',a',\tilde s)=0
        \end{array} \right.
\end{equation}
for all $s,s'\in\mc S$ and $a,a'\in\mc A$, then $\Dc{G(\xi')}{G(\xi)}=\Lc{\xi'}{\xi}$ for all~$\xi'\in\Xi$ and $\xi\in\Xi_0$. 
\end{proposition}

An elementary calculation reveals that $\theta=G(\xi)$ belongs to $\Delta((\mc S\times\mc A)^2)$ and satisfies~\eqref{eq:balanced-marginals} for any $\xi\in\Xi$. Therefore, $\Theta$ represents indeed the codomain of~$G$. Note also that if $\xi\in\Xi_0$ and if we define the policy~$\pi$ via~\eqref{eq:pi-from-xi}, then $G(\xi) ((s,a),(s',a')) = \pi(a'|s')\xi(s,a,s')$ for all $s,s'\in\mc S$ and $a,a'\in\mc A$.

\begin{proof}{\textbf{of Proposition~\ref{prop:Dmdp-Dmc-relation}}}
Select any~$\xi\in\Xi_0$ and~$\xi'\in \Xi$, and denote the stationary state-action and state distributions corresponding to~$\xi$ as~$\mu$ and~$\mu_\mc S$, respectively. Thus, we have
\[
    \mu(s,a)=\sum_{\tilde s\in\mc S} \xi(s,a,\tilde s) \quad \text{and} \quad \mu_{\mc S}(s)=\sum_{\tilde s\in\mc S} \sum_{\tilde a \in\mc A} \xi(s,\tilde a,\tilde s) \quad \forall s\in\mc S,\; a\in\mc A.
\]
The stationary state-action distribution~$\mu'$ and the stationary state distribution~$\mu'_\mc S$ corresponding to~$\xi'$ are defined analogously. We prove the claim first under the assumption that $\supp(\xi')\subseteq\supp(\xi)$. This implies that $\supp(G(\xi'))  \subseteq \supp(G(\xi)).$ 
By the definitions of $\mathsf{D_{mc}}$ and $G$, we thus have
\begin{align*}
	&\Dc{G(\xi')}{G(\xi)} 
    \\&=\sum_{s,s'\in\mc S, a,a'\in\mc A} \frac{\mu'(s',a')}{\mu'_{\mc S}(s')} \xi'(s,a,s') \Bigg( \log \Bigg(\frac{\mu'(s',a')}{\mu'_{\mc S}(s')} \frac{\xi'(s,a,s')}{\mu'(s,a)}\Bigg)
        - \log \Bigg(\frac{\mu(s',a')}{\mu_{\mc S}(s')} \frac{\xi(s,a,s')}{\mu(s,a)} \Bigg) \Bigg)
    \\&=\sum_{s,s'\in\mc S, a\in\mc A}  \xi'(s,a,s') \log  \frac{\xi'(s,a,s')}{\xi(s,a,s')} 
        +  \!\!\!\!
        \sum_{s'\in\mc S, a'\in\mc A} \mu'(s',a') \log \frac{\mu'(s',a')}{\mu(s',a')}
        - \!\!\!\!
        \sum_{s\in\mc S, a\in\mc A } \mu'(s,a) \log \frac{\mu'(s,a)}{\mu(s,a)}
        \\& \qquad-
        \sum_{s'\in\mc S } \mu'_{\mc S}(s')\log \frac{\mu'_{\mc S}(s')}{\mu_{\mc S}(s')}
    \\&=\sum_{s,s'\in\mc S, a\in\mc A}  \xi'(s,a,s') \log  \frac{\xi'(s,a,s')}{\xi(s,a,s')}
        -
        \sum_{s\in\mc S } \mu'_{\mc S}(s)\log \frac{\mu'_{\mc S}(s)}{\mu_{\mc S}(s)}
    \\&=\sum_{s,s'\in\mc S, a\in\mc A}\xi'(s,a,s')\left(\log \frac{\xi'(s,a,s')}{\sum_{\tilde s\in\mc S,\tilde a\in\mc A}\xi'(s,\tilde a,\tilde s)}-\log \frac{\xi(s,a,s')}{\sum_{\tilde s\in\mc S,\tilde a\in\mc A}\xi(s,\tilde a,\tilde s)}\right),
\end{align*}
where we have repeatedly used the convention that $0\log(0/q)=0$ for all $q\geq 0$. Assume next that $\supp(\xi')\nsubseteq \supp(\xi)$ such that $\supp(G(\xi'))\nsubseteq \supp(G(\xi))$. In this case, we find 
\[
    \Dc{G(\xi')}{G(\xi)}=\infty=\Lc{\xi'}{\xi}
\]
thanks to our convention that $p\log(p/0)=\infty$ for all $p>0$. Hence, the claim follows.
\end{proof}

Equipped with Proposition~\ref{prop:Dmdp-Dmc-relation}, we are now ready to prove that the empirical state-action-next-state distribution~$\qhat$ satisfies a large deviations principle with rate function $\mathsf{D}_{\mathsf{mdp}}$.

\begin{theorem}[Large deviations principle for MDPs]\label{thm:LDP:q}
For all $\xi\in \Xi_0$ and Borel sets $\mathcal{D} \subseteq \Xi$, the empirical state-action-next-state distribution $\qhat$ defined in~\eqref{MDP:estimator} satisfies
$$
\begin{aligned}
-\inf_{\xi'\in\mathrm{int}\mc D}\Lc{\xi'}{\xi} & \leq \liminf _{T \rightarrow \infty} \frac{1}{T} \log \mathbb{P}_\xi\left(\qhat \in \mathcal{D}\right) \\
& \leq \limsup_{T \rightarrow \infty} \frac{1}{T} \log \mathbb{P}_\xi\left(\qhat \in \mathcal{D}\right) \leq-\inf_{\xi'\in\mc D}\Lc{\xi'}{\xi}.
\end{aligned}
$$
\end{theorem}
\begin{proof}
Fix any~$\xi\in\Xi_0$, and set $\theta=G(\xi)$. One readily verifies that $\xi=F(\theta)$, where $F$ is the transformation defined in~\eqref{eq:F-definition}. Recall also that $\qhat=F(\thetahat)$, where~$\qhat$ and $\thetahat$ are the empirical estimators for~$\xi$ and~$\theta$ defined in~\eqref{MDP:estimator} and~\eqref{MC:estimator}, respectively. We already know from Theorem~\ref{thm:LDP} that~$\thetahat$ satisfies a large deviations principle with rate function~$\mathsf{D}_{\mathsf{mc}}$. By the contraction principle~\citep[Theorem~4.2]{dembo2011large}, which applies because~$F$ is continuous, $\qhat=F(\thetahat)$ thus satisfies a large deviations principle with rate function $I(\xi',\xi) =\inf_{\theta'\in\Theta,F(\theta')=\xi'}\Dc{\theta'}{\theta}.$

It remains to be shown that $I(\xi',\xi)= \Lc{\xi'}{\xi}$ for all $\xi\in\Xi_0$ and~$\xi'\in\Xi$. To this end, set again $\theta=G(\xi)$, and use~$\mu$ and~$\mu'$ to denote the stationary state-action distributions corresponding to~$\xi$ and~$\xi'$, respectively. For any $\theta'\in\Theta$ with $ F(\theta')=\xi'$ we then have
\begin{align*}
    \Dc{\theta'}{\theta}& =\D{\theta'}{\theta}-\D{\mu'}{\mu} \\
    &\ge \D{G(F(\theta'))}{G(F(\theta))}-\D{\mu'}{\mu}
    \\&= \D{G(\xi')}{G(\xi)}-\D{\mu'}{\mu}
    \\&=\Dc{G(\xi')}{G(\xi)}=\Lc{\xi'}{\xi}
\end{align*}
where the first equality exploits the definition of $\mathsf{D}_{\mathsf{mc}} $ (see Definition~\ref{def:conditional_relative_entropy}), and the inequality follows from the data processing inequality \citep[Lemma~3.11]{csiszar2011information}.
The second equality holds because $F(\theta')=\xi'$ by assumption and because $F(\theta)=F(G(\xi))=\xi$ for any~$\xi\in\Xi_0$. The third equality exploits again the definition of $\mathsf{D}_{\mathsf{mc}}$, and the fourth equality follows from Proposition~\ref{prop:Dmdp-Dmc-relation}. Hence, the infimum in the definition of $I(\xi',\xi)$ is attained by $\theta'= G(\xi')\in\Theta$. This in turn implies that $I(\xi',\xi) = \Lc{\xi'}{\xi}$. This observation completes the proof.
\end{proof}

The next corollary akin to Corollary~\ref{ldp:finite} establishes a finite-sample version of Theorem~\ref{thm:LDP:q}.

\begin{corollary}[Finite-sample version of Theorem~\ref{thm:LDP:q}]\label{ldp:finite:q}
For all $\xi \in \Xi_0$ and Borel sets $\mathcal{D}\subseteq \Xi$, the empirical doublet distribution $\qhat $ defined in~\eqref{MDP:estimator} satisfies
$$\begin{aligned}
\frac{1}{T}\log  \mathbb{P}_{\xi}\left(\qhat\in \mc D\right)
\leq \frac{1}{T}(\log T+\bar c+ d^2\log (T+1) )-\inf_{\xi'\in \mc D}{\Lc{\xi'}{\xi}} \quad \forall T \in \mb N,
\end{aligned}$$ 
where $\bar c>0$ is a universal constant that depends only on~$\xi$.
\end{corollary}
\begin{proof}
Fix any~$\xi\in\Xi_0$ and Borel set~$\mc D\subseteq \Xi$. Set $\theta=G(\xi)$, and let $F$ be the transformation defined in~\eqref{eq:F-definition}. As~$F$ is linear,  $F^{-1}(\mc D)$ is a Borel subset of~$\Theta$. 
As $\qhat=F(\thetahat)$, we may conclude that
\begin{align*}
    \frac{1}{T}\log\mathbb{P}_{\xi}\left(\qhat\in \mc D\right) & = \frac{1}{T}\log\mathbb{P}_{\xi}\left(\thetahat\in F^{-1}(\mc D)\right) \\
    &\leq \frac{1}{T}(\log T+\bar c+ d^2\log (T+1) )-\inf_{\theta'\in F^{-1}(\mc D)}\Dc{\theta'}{\theta} \\
    &=\frac{1}{T}(\log T+\bar c+ d^2\log (T+1) )-\inf_{\xi'\in\mc D}\; \inf_{\theta'\in\Theta, F(\theta')=\xi'}\Dc{\theta'}{\theta} \\
    &=\frac{1}{T}(\log T+\bar c+ d^2\log (T+1) )-\inf_{\xi'\in \mc D}{\Lc{\xi'}{\xi}} \quad \forall T\in\mb N,
\end{align*}
where the inequality follows from Corollary~\ref{ldp:finite}, and the last equality uses Proposition~\ref{prop:Dmdp-Dmc-relation}.
\end{proof}

The conditional relative entropy~$\mathsf{D}_{\mathsf{mdp}}$ will be instrumental for modeling distribution shifts. Besides representing the rate function of a large deviations principle, it inherits many useful properties from~$\mathsf{D}_{\mathsf{mc}}$ including level compactness, radial monotonicity and coercivity. The properties of~$\mathsf{D}_{\mathsf{mc}}$ are established in \citep[Proposition~5.1]{ref:Sutter-19}. The corresponding properties of~$\mathsf{D}_{\mathsf{mdp}}$ can be established similarly by adapting the proofs for~$\mathsf{D}_{\mathsf{mc}}$ in the obvious way.

\begin{lemma}[Properties of $\mathsf{D}_{\mathsf{mdp}}$]\label{lem:properties:Dm}
The following hold.
\begin{enumerate}[label = (\roman*)]
\item  \textbf{Lower semicontinuity.} $\Lc{\xi'}{\xi}$ is lower semicontinuous on~$\Xi\times\Xi$;\label{lem:continuity:Dm}
\item \textbf{Level-compactness.} $\{(\xi',\xi) \in\Xi\times \Xi: \Lc{\xi'}{\xi} \leq \rho\}$ is compact for every $\rho \geq 0$. \label{lem:level:compact:Dm}
\item \textbf{Coerciveness.} $\lim_{\zeta\in\Xi,\zeta\to\xi}\Lc{\xi'}{\zeta }=\infty$ for all $\xi'\in\Xi_0$ and $\xi\in\Xi$ with $\supp(\xi')\nsubseteq \supp(\xi)$.\label{lem:coerc:Dm}
\item
\textbf{Radial monotonicity in $\xi$}. $
\cl\{\xi \in \Xi_0: \Lc{\xi'}{\xi}<\rho\}=\{\xi \in \Xi: \Lc{\xi'}{\xi}\leq \rho\}$ for all $\xi' \in \Xi$ and $\rho>0$;\label{lem:radial:monotone:Dm}
\item \textbf{Continuity of the sublevel set mapping.} The set-valued mapping $\Gamma:\Xi \rightrightarrows \Xi$ defined through $\Gamma(\xi')=\{ \xi\in\Xi :\Lc{\xi'}{\xi}\leq \rho\}$ is continuous in~$\xi'\in \Xi$ for every $\rho>0$. \label{Dm:set:continuous}
\item \textbf{Convexity in~$\xi'$.} $\Lc{\xi'}{\xi}$ is convex in~$\xi'\in\Xi$ for every fixed $\xi\in\Xi$. \label{Dm:convexity}
\end{enumerate}
\end{lemma}

\begin{proof}
Assertion~\ref{lem:continuity:Dm} follows from the definition of~$\mathsf{D}_{\mathsf{mdp}}$. Assertions~\ref{lem:level:compact:Dm},~\ref{lem:radial:monotone:Dm}, and~\ref{Dm:convexity} can be shown by adapting the proof of~\citep[Proposition 5.1]{ref:Sutter-19} from~$\mathsf{D}_{\mathsf{mc}}$ to~$\mathsf{D}_{\mathsf{mdp}}$.
Assertion~\ref{lem:coerc:Dm} holds because $\lim_{p\to 0} p\log(p/q)=0$ for any $q\ge 0$ and $\lim_{p \to 0} q\log(q/p)=\infty$ for any~$p>0$. Assertion~\ref{Dm:set:continuous} can be shown by adapting the proof of~\citep[Proposition~3.1]{ref:Sutter-19} from~$\mathsf{D}_{\mathsf{mc}}$ to~$\mathsf{D}_{\mathsf{mdp}}$. 
\end{proof}

\section{Distributionally Robust Off-Policy Evaluation}
\label{ssec:OPE:dist:shift}

Fix any MDP $(\mc S,\mc A, Q,r,\eta)$ of the type studied in Section~\ref{ssec:problem:statement}, and assume that~$Q\in\mc Q_0$. In addition, fix any stationary policy $\pi\in\Pi_0$. Such policies are called {\em exploratory} because $\pi>0$. We know from Section~\ref{ssec:problem:statement} that the state-action pairs of the MDP follow an irreducible Markov chain with a transition probability matrix of the form~\eqref{expr:P:from:pi:Q}. By the Markov law of large numbers \citep[Theorem~1.7.6]{norris1998markov}, the state-action-next-state distribution~$\xi\in\Xi_0$ of this Markov chain satisfies
\begin{align}
    \label{eq:general-xi-definition}
    \xi(s,a,s')= \lim_{T\to\infty} \frac{1}{T} \sum_{t=1}^T\mb P_{\pi,Q} \left(S_t=s,\;A_t=a,\;S_{t+1}=s'\right)\quad\forall s,s'\in\mc S,\; a\in\mc A,
\end{align}
where $\mb P_{\pi,Q}$ is defined as in Section~\ref{ssec:problem:statement}. In addition, $\xi$ is related to the stationary state-action distribution~$\mu$ through~\eqref{xi:alt:pi:Q}. We define the long-run average reward generated by~$\pi$ under~$Q$ as
\begin{align*}
    V(\xi)= \lim _{T \rightarrow \infty} \mb E_{\pi,Q} \left[ \frac{1}{T} \sum_{t=1}^{T}r(S_t, A_t) \right].
\end{align*}
Note that $V(\xi)$ is independent of~$\eta$ and depends on $\pi$ and~$Q$ only indirectly through~$\xi$ because
\begin{align*}
    V(\xi)= \lim _{T \rightarrow \infty} \sum_{s,s'\in\mc S} \sum_{a\in\mc A} r(s,a) \frac{1}{T} \sum_{t=1}^{T}  \mb P_{\pi,Q} \left(S_t=s,\;A_t=a,\; S_{t+1}=s' \right)
    =\sum_{s,s'\in\mc S} \sum_{a\in\mc A} r(s,a) \xi(s,a,s').
\end{align*}
This shows in particular that $V(\xi)$ is linear and thus continuous in~$\xi$. Note also that since~$\xi\in\Xi_0$, we have that $V(\xi)= \lim _{T \rightarrow \infty} \frac{1}{T} \sum_{t=1}^{T}r(S_t, A_t)$ $\mb P_\xi$-almost surely by virtue of the Markov law of large numbers \citep[Theorem~1.7.6]{norris1998markov}. We express the long-run average reward~$V(\xi)$ as a function of~$\xi$ instead of~$\mu$ because $\xi$ contains full information about~$\pi$ and~$Q$, whereas~$\mu$ does not.

Assume now that~$Q$ is {\em unknown}, which is a standard assumption in reinforcement learning. In contrast, the reward function~$r$ and the initial state distribution~$\eta$ are {\em known}. We also distinguish an {\em unknown} behavioral policy~$\pi_0\in\Pi_0$ and a {\em known} evaluation policy~$\pi\in\Pi_0$. In addition, we refer to the state-action-next-state distributions associated with~$\pi_0$ and~$\pi$ as~$\xi_0\in\Xi_0$ and~$\xi\in\Xi$, respectively. Note that both~$\xi_0$ and~$\xi$ are unobservable because they depend on the unknown transition kernel~$Q$. 

In the following we assume to have access to a single state-action trajectory $X_1,\ldots,X_T$ generated under the behavioral policy~$\pi_0$. This trajectory provides information about the unknown transition kernel~$Q$ and can thus be used to estimate the reward function at~$\xi_0$ as well as at~$\xi$. We thus distinguish two fundamental estimation problems. The {\em on-policy evaluation problem} asks for an estimate of~$V(\xi_0)$, that is, the long-run average reward of the behavioral policy~$\pi_0$ that generates the data. In contrast, the {\em off-policy evaluation problem} asks for an estimate of~$V(\xi)$, that is,  the long-run average reward of the evaluation policy~$\pi$. Note that there is {\em no} data generated under~$\pi$. The off-policy evaluation problem can be viewed as an estimation problem with a distribution shift because we aim to estimate the expected reward under~$\mb P_\xi$ from data generated under~$\mb P_{\xi_0}$.


We will see below that both~$\xi$ as well as~$\te$ admit asymptotically consistent estimators. A na\"ive solution of the on-policy evaluation problem would be to approximate~$V(\xi_0)$ by the plug-in estimator~$V(\qhat)$, where $\qhat$ is the empirical distribution defined in~\eqref{MDP:estimator}. By \citep[Theorem~1.7.6]{norris1998markov}, $\qhat$ converges almost surely to~$\xi_0$ as~$T$ grows because the Markov chain of state-action pairs is irreducible under~$\pi_0$. Hence, $V(\qhat)$ converges $\mb P_{\xi_0}$-almost surely to~$V(\xi_0)$. However, we will later see that~$V(\qhat)$ is likely to overestimate the true average reward---especially at small sample sizes. Also, it does not admit a natural generalization to off-policy evaluation. An important ingredient for solving the off-policy evaluation problem is the following distribution shift function.

\begin{definition}[Distribution shift function]\label{def:f}
For every exploratory evaluation policy~$\pi\in\Pi_0$, the distribution shift function $\fe:\Xi_0\to\Xi$ satisfies~$f_\pi(\xi_0)=\xi$, where $\xi$ is the state-action-next-state distribution corresponding to $(\pi,Q)$ and where $Q$ is the transition kernel induced by~$\xi_0$ through~\eqref{eq:q-from-xi}.
\end{definition}

Note that~$f_\pi$ is well-defined because the evaluation policy~$\pi\in\Pi_0$ is given and~$Q\in\mc Q_0$ is determined by~$\xi_0\in\Xi_0$ via~\eqref{eq:q-from-xi} and because we know from Section~\ref{ssec:problem:statement} that $(\pi,Q)$ induces a unique~$\xi\in\Xi_0$. Using the distribution shift function, we can recast the long-run average reward of the evaluation policy as $V(\xi)=V(f_\pi(\xi_0))$.
%
%
A na\"ive solution to the off-policy evaluation problem would therefore be to approximate~$V(\xi)$ by the plug-in estimator~$V(\fe(\qhat))$. This simply amounts to replacing the unknown transition kernel~$Q$ with its maximum likelihood estimator, which is obtained by substituting~$\qhat$ into~\eqref{eq:q-from-xi}. In the discounted reward setting this approach is sometimes referred to as the {\em direct method}. We emphasize that if~$T$ is small, then~$\qhat$ may adopt values in~$\Xi\backslash \Xi_0$, and in these cases the direct estimator~$V(\fe(\qhat))$ is undefined. As~$\qhat$ converges almost surely to~$\xi_0\in\Xi_0$, however, $V(\fe(\qhat))$ is eventually well-defined for all sufficiently large~$T$. Moreover, $V(\fe(\qhat))$ converges almost surely to~$V(\xi)$ because~$V$ and~$f_\pi$ are continuous.

\begin{lemma}[Continuity of $\fe$]
\label{lem:cont:sol:i-proj}
The distribution shift function $\fe$ is continuous on $\Xi_0.$ 
\end{lemma}

\begin{proof}
By Definition~\ref{def:f}, the distribution shift function satisfies $f_\pi(\xi_0)=\xi$, and~\eqref{xi:alt:pi:Q} implies that
\begin{align*}
    \xi(s,a,s')= Q(s'|s,a)\mu(s,a) \quad \forall s,s'\in\mathcal{S}, ~ a,a'\in\mc A.
\end{align*}
By~\eqref{eq:q-from-xi}, the transition kernel~$Q$ is a rational and therefore continuous function of~$\xi_0\in \Xi_0$. Hence, the transition probability matrix $P$ defined via~\eqref{expr:P:from:pi:Q} is also continuous in~$\xi_0$. By the Perron-Frobenius theorem, the stationary distribution~$\mu$ is the unique positive solution of the linear equations
\[
    \sum_{s\in\mc S, a\in\mc A}\mu(s,a)=1\quad \text{and} \quad \mu(s',a')=\sum_{s\in\mc S, a\in\mc A}\mu(s,a) P((s,a),(s',a'))\quad\forall s'\in\mc S,~a'\in\mc A.
\]
Therefore it inherits continuity in~$\xi_0$ from~$P$. In summary, these insights imply that $f_\pi(\xi_0)=\xi$ is continuous in~$\xi_0$ throughout~$\Xi_0$. Hence, the claim follows.
\end{proof}

The following extension of Lemma~\ref{lem:cont:sol:i-proj} will be useful in Section~\ref{ssec:OPL:dist:shift}.

\begin{corollary}[Continuity of $f_\pi$]
\label{cor:joint-cont:i-proj}
The function $\fe(\xi_0)$ is continuous in $(\pi,\xi_0)\in\Pi_0\times\Xi_0$. 
\end{corollary}
\begin{proof}
By~\eqref{eq:q-from-xi}, $Q$ is independent of~$\pi$, and by~\eqref{expr:P:from:pi:Q}, $P$ is continuous in~$(\pi,\xi_0)$. The continuity of $\fe(\xi_0)$ can thus be shown by repeating the proof of Lemma~\ref{lem:cont:sol:i-proj} with obvious modifications.
\end{proof}

Lemma~\ref{lem:cont:sol:i-proj} implies via the Markov law of large numbers \citep[Theorem~1.7.6]{norris1998markov} that the direct estimator~$V(\fe(\qhat))$ converges $\mb P_{\xi_0}$-almost surely to $V(f_\pi(\xi_0))=V(\xi)$ and is therefore asymptotically consistent. We now introduce an alternative---distributionally robust---estimator~$V^{\pi}_\rho(\qhat)$ for~$V(\xi)$, which is defined through the worst-case value function $V^{\pi}_\rho: \Xi\to\mb R$ with
\begin{align}
    \label{def:vstar:q}
    V^{\pi}_\rho(\xi')=\inf _{\xi_0\in \Xi_0}\left\{V(\fe(\xi_0)):\Lc{\xi'}{\xi_0} \leq \rho\right\}\quad \forall \xi'\in\Xi.
\end{align}
By construction, the optimization problem in~\eqref{def:vstar:q} seeks the worst-case reward of the evaluation policy~$\pi$ with respect to all state-action-next-state distributions~$\xi_0$ close to a realization~$\xi'\in\Xi$ of the empirical estimator~$\qhat$. Here, proximity between~$\xi'$ and~$\xi_0$ is measured by the conditional relative entropy for MDPs introduced in Definition~\ref{def:Lc}.
Hence, the feasible set of problem~\eqref{def:vstar:q} can be viewed as a conditional relative entropy ball of radius $\rho\geq 0$ in the space of state-action-next-state distributions. If $\rho=0$, then the distributioinally robust estimator~$V^{\pi}_\rho(\qhat)$ collapses to the direct estimator~$V(\fe(\qhat))$, which is ill-defined unless~$\qhat\in\Xi_0$. If~$\rho>0$, on the other hand, then problem~\eqref{def:vstar:q} is guaranteed to be feasible for any~$\xi'\in\Xi$. Indeed, $\{\xi \in \Xi: \Lc{\xi'}{\xi}\leq \rho\}$ is non-empty because $\Lc{\xi'}{\xi'}=0$. This implies via the radial monotonicity established in Lemma~\ref{lem:properties:Dm}\,\ref{lem:radial:monotone:Dm} that problem~\eqref{def:vstar:q} is feasible. Hence, $V^{\pi}_\rho(\qhat)$ is well-defined for all $\rho>0$.

In addition, if~$\rho>0$, then the minimum in~\eqref{def:vstar:q} is attained whenever~$\xi'\in\Xi_0$. Indeed, the objective function~$V(\fe(\xi_0))$ is continuous on~$\Xi_0$ thanks to Lemma~\ref{lem:cont:sol:i-proj}, and the feasible set satisfies 
\begin{align}
    \label{eq:Gamma-Xi-Xi0}
    \big\{\xi_0\in\Xi_0 : \Lc{\xi'}{\xi_0} \leq \rho\big\}=\big\{\xi_0\in\Xi : \Lc{\xi'}{\xi_0} \leq \rho\big\}.
\end{align}
The equality holds because $\xi'\in\Xi_0$, which implies via Lemma~\ref{lem:properties:Dm}\,\ref{lem:coerc:Dm} that $\lim_{\zeta\in\Xi,\zeta\to\xi_0}\Lc{\xi'}{\zeta }=\infty$ for all $\xi_0\in\Xi\backslash \Xi_0$. Hence the feasible set is compact by virtue of Lemma~\ref{lem:properties:Dm}\,\ref{lem:level:compact:Dm}, and the infimum in~\eqref{def:vstar:q} is attained by the Weierstrass extreme value theorem. From now on we assume that~$\rho>0$.

In the remainder of this section we will show that the proposed distributionally robust estimator~$V^{\pi}_\rho(\qhat)$ for~$V(\fe(\xi_0))$ enjoys rigorous finite-sample and asymptotic consistency guarantees. In addition, we show that $V^{\pi}_\rho(\qhat)$ is, in a precise sense, the least conservative estimator whose out-of-sample disappointment decays exponentially at rate~$\rho$. Throughout this section, we restrict attention to estimators of the form~$\widehat V^\pi (\qhat) $, where $\widehat V^\pi:\Xi\to\mb R$ is an arbitrary lower semicontinuous function. We refer to the probability $\mb P_{\xi_0}(V(\fe({\xi_0}))<\widehat V^\pi (
\qhat))$ as the \textit{out-of-sample disappointment} of the estimator~$\widehat V^\pi( \qhat)$ under the model~${\xi_0}\in\Xi_0$. It quantifies the probability that the actual expected reward of the evaluation policy is strictly smaller than the reward predicted by the estimator. If the out-of-sample disappointment is large, then $\widehat V^\pi (\qhat)$ {\em overestimates} the expected reward $V(\fe({\xi_0}))$ with high probability. Hence, the estimator is overly optimistic, which may lead to disappointment.
In the following, we will restrict our attention to estimators~$\widehat V^\pi(\qhat)$ that satisfy
\begin{equation}\label{eq:condition:admissibility}
    \limsup_{T\to\infty}\frac{1}{T} \log\mb P_{\xi_0}\left( V(\fe({\xi_0})) <  \widehat V^\pi(\qhat)\right) \leq -\rho \quad \forall {\xi_0}\in\Xi_0.
\end{equation}
This condition ensures that the out-of-sample disappointment decays asymptotically as $e^{-\rho T+o(T)}$. It implies that $\widehat{V}^\pi (\qhat)$ becomes an increasingly reliable lower confidence bound on $V(\fe({\xi_0}))$ as the sample size grows. Underestimating the true expected reward~$V(\fe({\xi_0}))$ means to err on the side of caution. We will now demonstrate that the distributionally robust estimator $V^{\pi}_\rho(\qhat)$ satisfies~\eqref{eq:condition:admissibility} and that it is in fact the least conservative estimator satisfying~\eqref{eq:condition:admissibility}. As a preparation, the next proposition shows that the function~$V^{\pi}_\rho$ is lower semicontinuous on~$\Xi$ and continuous on~$\Xi_0$.
\begin{lemma}[Continuity properties of $V^{\pi}_\rho$]
    \label{lem:lsc}
    For any fixed $\rho>0$ and $\pi\in\Pi_0$, the distributionally robust value function $V^{\pi}_\rho(\xi')$ is lower semicontinuous in~$\xi'\in\Xi$ and continuous in~$\xi'\in\Xi_0$.
\end{lemma}
\begin{proof}
Define $\overline{V}(\xi_0)=
\lim_{\delta \downarrow 0}\inf_{\zeta\in\Xi_0} \{V(\fe(\zeta)):\|\zeta-\xi_0\|\le\delta\}$
as the unique lower semicontinuous extension of~$V(\fe(\xi_0))$ to $\Xi$. Indeed, $\lsc$ is lower semicontinuous by construction and coincides with~$V(\fe(\xi_0))$ on~$\Xi_0$ by virtue of Lemma~\ref{lem:cont:sol:i-proj}. Define now the set-valued mapping $\Gamma(\xi')=\{ \xi_0\in\Xi :\Lc{\xi'}{\xi_0}\leq \rho\}$, which is continuous thanks to Lemma~\ref{lem:properties:Dm}\,\ref{Dm:set:continuous}.
Noting that $\Gamma(\xi')\ne\emptyset$ for every $\xi'\in\Xi$ because $\Lc{\xi'}{\xi'}=0$, \citep[Lemma~17.29]{aliprantis2006infinite} implies that
\begin{align*}
    \varphi(\xi') = \min_{\xi_0\in\Gamma(\xi')} \lsc(\xi_0)
\end{align*}
is lower semicontinuous on $\Xi$. In addition, we know from~\eqref{eq:Gamma-Xi-Xi0} that $\Gamma(\xi')\subseteq\Xi_0$ for every $\xi'\in\Xi_0$. 
As $\lsc(\xi_0) = V(\fe(\xi_0))$ is continuous on $\Xi_0$ by virtue of Lemma~\ref{lem:cont:sol:i-proj}, we may then use~\citep[Theorem~17.31]{aliprantis2006infinite} to conclude that $\varphi(\xi')$ is continuous on~$\Xi_0$. The claim thus follows if we can show that $V_\rho^\pi(\xi') = \varphi(\xi')$ for every~$\xi'\in\Xi$. By~\eqref{eq:Gamma-Xi-Xi0}, this identity clearly holds for every~$\xi'\in\Xi_0$. Assume now that~$\xi'\in\Xi\backslash\Xi_0$, and select any $\xi_0^\star \in\arg\min_{\xi_0\in\Gamma(\xi')} \lsc(\xi_0)$, which exists because $\Gamma(\xi')$ is non-empty and compact and because $\lsc(\xi_0)$ is lower semicontinuous. 
Thus, we find
\[
    \varphi(\xi')\leq V_\rho^\pi(\xi') \leq \lim_{\delta \downarrow 0}\inf_{\zeta\in\Xi_0} \left\{V(\fe(\zeta)):\|\zeta-\xi^\star_0\|\le\delta \right\} = \lsc(\xi_0^\star) = \varphi(\xi'),
\]
where the first inequality holds because the feasible set of~\eqref{def:vstar:q} is obtained by restricting~$\Gamma(\xi')$ to~$\Xi_0$, and the two equalities follow from the definitions of~$\lsc$ and~$\xi_0^\star$, respectively. Hence, the identity $V_\rho^\pi(\xi') = \varphi(\xi')$ holds indeed for all~$\xi'\in\Xi$. This observation completes the proof.
\end{proof}

The following extension of Lemma~\ref{lem:lsc} will be useful in Section~\ref{ssec:OPL:dist:shift}.
\begin{corollary}[Continuity properties of $V^\pi_\rho$]
\label{cor:admissible:dr-policy}
For any fixed~$\rho>0$, the distributionally robust value function $V^\pi_\rho(\xi')$ is continuous in~$\pi$ and lower semicontinuous in $\xi'$ throughout $\Pi_\epsilon\times\Xi$. 
\end{corollary}
\begin{proof}
The proof parallels that of Lemma~\ref{lem:lsc} with obvious minor modifications ({\em e.g.}, Corollary~\ref{cor:joint-cont:i-proj} must be used instead of Lemma~\ref{lem:cont:sol:i-proj}).
\end{proof}

We are now ready to prove that the out-of-sample disappointment of the distributionally robust estimator $V^{\pi}_\rho(\qhat)$ decays exponentially at rate~$\rho$ with the sample size~$T.$

\begin{theorem}[Out-of-sample disappointment of $V^{\pi}_\rho(\qhat)$]\label{thm:gen:bound}
For every $\rho>0$, $\pi\in\Pi_0$ and $\xi_0\in\Xi_0$, the distributionally robust estimator $V^{\pi}_\rho(\qhat)$ satisfies
\begin{align*}
    \limsup_{T\to\infty} \frac{1}{T}\log \mathbb{P}_{\xi_0}\left(V(f_\pi(\xi_0)) < V^{\pi}_\rho(\qhat)\right) \le -\rho. 
\end{align*}
\end{theorem}
Theorem~\ref{thm:gen:bound} asserts that the out-of-sample disappointment of~$V^{\pi}_\rho(\qhat)$ decays as~$e^{-\rho T+o(T)}.$ Its proof is omitted because it is a direct consequence of the following finite-sample guarantee.
\begin{theorem}[Finite-sample guarantee for $\!V^{\pi}_\rho(\qhat)$]
\label{thm:finite:sample}
For every $\xi_0\in\Xi_0$ there is $\bar c>0$ such that the distributionally robust estimator $V^{\pi}_\rho(\qhat)$ satisfies the following for all $\rho> 0$, $\pi\in\Pi_0$ and~$T \in \mb N$.
\begin{align*}
    \frac{1}{T}\log  \mb P_{\xi_0}\left(V(f_\pi(\xi_0)) < V^{\pi}_\rho(\qhat)\right) \leq \frac{1}{T}(\log (T)+\bar c+ (SA)^2\log (T+1) )-\rho
\end{align*}
\end{theorem}
\begin{proof}
Define the disappointment set $\mc D=\{\xi'\in\Xi: V(\fe(\xi_0))<V^{\pi}_\rho(\xi')\}$, which is open because $V^{\pi}_\rho(\xi')$ is lower semicontinuous by virtue of Lemma~\ref{lem:lsc}. Then, we have
\begin{align*}
    \frac{1}{T}\log  \mb P_{\xi_0} \left(V(f_\pi(\xi_0)) < V^{\pi}_\rho(\qhat)\right)
    &=  \frac{1}{T}\log \mb P_{\xi_0}\left(\qhat\in\mc D\right)\\
    &\le\frac{(SA)^2}{T}\log (T+1) + \frac{1}{T}(\log (T)+\bar c)-\inf_{\xi'\in\mc D}\Lc{\xi'}{\xi_0}\\
    &\leq \frac{1}{T}((SA)^2\log (T+1)+\log (T)+\bar c  )-\rho \quad   \forall T\in\mb N,
\end{align*}
where the first inequality follows from Corollary~\ref{ldp:finite:q}, and the second inequality holds because $\Lc{\xi'}{\xi_0}>\rho$ for any $\xi'\in\mc D$. Indeed, the definition of~$V^{\pi}_\rho(\xi')$ in~\eqref{def:vstar:q} readily implies that $V(\fe(\xi_0))\ge V^{\pi}_\rho(\xi')$ whenever $\Lc{\xi'}{\xi_0}\le\rho.$ By contraposition, this means that $\Lc{\xi'}{\xi_0}>\rho$ whenever $V(\fe(\xi_0))<V^{\pi}_\rho(\xi')$. Recall also from Corollary~\ref{ldp:finite:q} that $\bar c$ depends only on~$\xi_0$.
\end{proof}

We now show that the distributionally robust predictor $V^{\pi}_\rho(\qhat)$ is efficient in the sense that it represents the least conservative estimator whose out-of-sample disappointment decays at rate~$\rho$.

\begin{theorem}[Statistical efficiency of $V^{\pi}_\rho(\qhat)$]\label{thm:pareto:vstar}
For every $\rho>0$, $\pi\in\Pi_0$ and lower semicontinuous function~$\widehat{V}^\pi:\Xi\to\mb R$ such that the corresponding reward estimator $\widehat V^\pi(\qhat)$ satisfies the out-of-sample guarantee~\eqref{eq:condition:admissibility}, we have $V^\pi_\rho(\xi')\geq \widehat V^\pi(\xi')$ for all~$\xi'\in \Xi$.
%
\end{theorem}
 
\begin{proof}
We first show that the claim holds for all $\xi'\in\Xi_0$. Assume for the sake of contradiction that there exists a lower semicontinuous function $\widehat{V}^\pi:\Xi\to\mb R$ with $\widehat V^\pi(\qhat)$  satisfying~\eqref{eq:condition:admissibility} and an estimator realization $\xi_0'\in\Xi_0$ with $\widehat{V}^\pi ({\xi_0'})> V^{\pi}_\rho({\xi_0'})$,  and define $\epsilon_0= \widehat V^\pi({\xi_0'})-V^{\pi}_\rho({\xi_0'})>0$. Next, let $\xi^\star_1\in\Xi_0$ be a minimizer of problem~\eqref{def:vstar:q} for $\xi'=\xi_0'$, which exists because $\xi_0'\in\Xi_0$. We thus have $\Lc{\xi_0'}{\xi^\star_1}\le\rho$ by feasibility and $V^{\pi}_\rho(\xi_0')=V(\fe(\xi^\star_1))$ by optimality. The radial monotonicity of $\mathsf{D_{mdp}}$ established in~Lemma~\ref{lem:properties:Dm}\ref{lem:radial:monotone:Dm} further guarantees that there exists a sequence $\{\xi^\star_k\}_{k\in\mb N}$ in $\Xi_0$  such that $\Lc{\xi_0'}{\xi^\star_k}<\rho $ for all $k\in\mb N$ and $\lim_{k\to\infty}\xi^\star_k=\xi^\star_1$. 
In addition, the continuity of $V(\fe(\xi^\star_1))$ on~$\Xi_0$ implies that there exists a model $\xi^\star_0\in\Xi_0$ with $V(\fe(\xi^\star_0))<V(\fe(\xi^\star_1))+\epsilon_0/2$ and $\rho_0=\Lc{\xi_0'}{\xi^\star_0}<\rho.$ In summary, we have thus shown that
\begin{align*}
    V^{\pi}_\rho(\xi_0')>V(\fe(\xi^\star_1))-\frac{\epsilon_0}{2} >V(\fe(\xi^\star_0))-\epsilon_0,
\end{align*}
which allows us to conclude that
\begin{align}\label{ineq:optimality-proof-predictor}
    \widehat V^\pi({\xi_0'})=V^{\pi}_\rho({\xi_0'})+\epsilon_0>V(\fe(\xi^\star_0)).
\end{align}
Next, we introduce the disappointment set 
$\mc D=\{\xi'\in\Xi: V(\fe(\xi^\star_0))<\widehat V^\pi(\xi')\}$ and observe that ${\xi_0'}\in\mc D.$ As~$\widehat V^\pi $ is lower semicontinuous  on $\Xi$ by assumption, the set $\Xi\backslash\mc D=\{\xi'\in\Xi: V(\fe(\xi^\star_0))\ge\widehat V^\pi(\xi')\}$ is closed, which in turn implies that $\mc D$ is open. We thus find
\begin{align*}
  \limsup_{T\to\infty} \frac{1}{T}\log  \mathbb{P}_{\xi^\star_0}\left(V(\fe(\xi^\star_0))< \widehat V^\pi(\qhat)\right) 
  &=  \limsup_{T\to\infty} \frac{1}{T}\log \mb P_{\xi^\star_0}\big(\qhat\in\mc D\big)\\
 & \geq -\inf_{\xi'\in\mathrm{int} \mc D} \mathsf{D_{mdp}}(\xi'\|\xi^\star_0)\geq -\rho_0>-\rho,
  \end{align*}
where the first inequality follows from Theorem~\ref{thm:LDP:q}, and the second inequality holds because $\xi_0'\in \mc D=\mathrm{int} \mc D$ and $ \mathsf{D_{mdp}}(\xi_0'\|\xi^\star_0) = \rho_0.$  Hence, the out-of-sample disappointment of~$\widehat V^\pi(\qhat)$ fails to decay at rate~$\rho$, which contradicts our initial assumption. Thus, the claim follows.

Next, assume that~$\xi'\in\Xi\backslash\Xi_0$. For any $n\in\mb N$, let $\xi'_n\in\Xi_0$ be a $1/n$-optimal solution of
\begin{align}
    \label{eq:auxiliary-inf-problem}
    \inf_{\zeta'\in\Xi_0} \{\widehat V^\pi(\zeta') : \|\zeta' - \xi'\|_2 \le 1/n \}.   
\end{align}
The feasibility of $\xi_n'$ in~\eqref{eq:auxiliary-inf-problem} implies that $\|\xi'_n-\xi'\|_2\le 1/n$ such that $\lim_{n\to\infty} \xi_n' = \xi'$. Thus, we have
\begin{equation}
\label{eq:vhatpi-estimate}
\begin{aligned}
   \widehat V^\pi(\xi') \le \liminf_{n\to\infty} \widehat V^\pi(\xi_n') &=  \lim_{n\to\infty} \inf_{\zeta'\in\Xi_0} \{ \widehat V^\pi(\zeta')  :  \|\zeta' - \xi'\|_2 \le 1/n\} \\
   &\le \lim_{n\to\infty}\inf_{\zeta'\in\Xi_0}\{ V_\rho^\pi(\zeta')  :  \|\zeta' - \xi'\|_2 \le 1/n\}\\
   &= \lim_{n\to\infty} \inf_{\zeta', \zeta\in\Xi_0}\{ V(f_\pi(\zeta))  :  \Lc{\zeta'}{\zeta}\le \rho,\, \|\zeta' - \xi'\|_2 \le 1/n\},
\end{aligned}
\end{equation}
where the first inequality follows from the assumed lower semicontinuity of~$\widehat V^\pi$. The equality holds because~$\xi_n'$ is a $1/n$-optimal solution of~\eqref{eq:auxiliary-inf-problem}, and the second inequality follows from the first part of the proof, where we have shown that $\widehat V^\pi(\zeta') \le V_\rho^\pi(\zeta') $ for all $\zeta'\in\Xi_0$. The second equality, finally, exploits the definition of~$V^\pi_\rho(\zeta')$.
Fix now any $\epsilon>0$, and let~$\zeta_\epsilon\in\Xi_0$ be an $\epsilon$-optimal solution of~\eqref{def:vstar:q}. Also, set $\zeta'_{\epsilon,n}=(1-1/(2n))\xi'+1/(2n)\zeta_\epsilon$, and note that $\zeta'_{\epsilon,n}\in\Xi_0$ because~$\zeta'_{\epsilon,n}$ represents a strict convex combination of $\xi'\in\Xi$ and $\zeta_\epsilon\in\Xi_0$. By construction, we thus have
\[
    \|\zeta'_{\epsilon,n}-\xi'\|_2 =\| \zeta_\epsilon - \xi'\|_2/(2n) \le (\|\xi'\|_2 + \|\zeta_\epsilon\|_2 )/(2n) \le 1/n,
\]
where the last inequality holds because $\zeta_\epsilon, \xi' \in \Delta(\mc S\times \mc A\times \mc S)$.
In addition, we find 
\[
    \Lc{\zeta'_{\epsilon,n}}{\zeta_\epsilon}
    \le (1-1/(2n))\Lc{\xi'}{\zeta_\epsilon} + \Lc{\zeta_\epsilon}{\zeta_\epsilon}/(2n) \le \rho ,
\]
where the first inequality follows from the convexity of the conditional relative entropy in its first argument (see Lemma~\ref{lem:properties:Dm}\,\ref{Dm:convexity}), and the second inequality holds because~$\zeta_\epsilon$ is feasible in~\eqref{def:vstar:q}. In summary, we have shown that $(\zeta'_{\epsilon,n},\zeta_\epsilon)$ constitutes a feasible solution for the last minimization problem in~\eqref{eq:vhatpi-estimate}. This observation readily implies that
\begin{equation*}
\begin{aligned}
   \widehat V^\pi(\xi') & \le V(f_\pi(\zeta_\epsilon)) \le V^\pi_\rho(\xi') + \epsilon.
\end{aligned}
\end{equation*}
where the second inequality holds because $\zeta_\epsilon$ is an~$\epsilon$-optimal solution of~\eqref{def:vstar:q}. As this estimate holds for every~$\epsilon>0$, we may finally conclude that $\widehat V^\pi(\xi')\le V^\pi_\rho(\xi')$. Hence, the claim follows.
\end{proof} 

Theorem~\ref{thm:pareto:vstar} is inspired by~\citep[Theorem 3.1]{ref:Sutter-19}, which establishes a similar result for general data-generating processes that are {\em not} affected by a distribution shift. On the technical level, Theorem~\ref{thm:pareto:vstar} is also more general because it establishes the optimality of the distributionally robust estimator within the class of estimators of the form $\widehat V^\pi(\qhat)$ induced by lower semicontinuous functions $\widehat V^\pi$, whereas \citep[Theorem 3.1]{ref:Sutter-19} focuses exclusively on estimators induced by continuous functions. Theorem~\ref{thm:pareto:vstar} also  implies that $V^\pi_\rho(\qhat)\geq \widehat V^\pi(\qhat)$ for all~$T\in\mb N$ because $V^\pi_\rho(\xi')\geq \widehat V^\pi(\xi')$ holds also for all~$\xi'\in \Xi\backslash\Xi_0$. In contrast, the optimality result in~\citep[Theorem 3.1]{ref:Sutter-19} is asymptotic, that is, it holds only in the limit when~$T$ tends to infinity.

The following proposition inspired by \citep[Theorem 3]{li2021distributionally} shows that the estimator~$V^{\pi}_\rho$ becomes asymptotically consistent if the radius~$\rho$ of the ambiguity set shrinks with~$T$.

\begin{proposition}[Asymptotic consistency  of $V^{\pi}_\rho(\qhat)$]\label{thm:asympt:consistency}
\!If $\{\rho_T\}_{T=1}^\infty$ is a sequence of non-negative radii with $\lim_{T\to\infty} \rho_T=0$, then
\begin{align}
    \lim_{T\to\infty} V^{\pi}_{\rho_T}(\qhat)
    =V({\fe(\xi_0)}) \quad \mathbb{P}_{\xi_0}\text{-}a.s. \quad \forall {\xi_0}\in\Xi_0.
\end{align}
\end{proposition}
\begin{proof}
Fix any ${\xi_0}\in\Xi_0$ and any sample~$\omega\in\Omega$ for which $\qhat(\omega)\in\Xi$ converges to $\xi_0$ as~$T$ grows. For this particular sample~$\omega$, one can proceed as in the proof of~\citep[Theorem 3]{li2021distributionally} to show that problem~\eqref{def:vstar:q} with $\xi'=\qhat(\omega)$ is solvable for all sufficiently large~$T$ and that the corresponding minimizer~$\xi_{T}^{\star}(\omega)\in\Xi_0$ converges to~$\xi_0$ as $T$ grows. This allows us to conclude that 
$$
\lim_{T \rightarrow \infty} V^{\pi}_{\rho_T}(\qhat(\omega))=\lim _{T \rightarrow \infty}V(\fe(\xi_{T}^{\star}(\omega)))=V(\fe({\xi_0})),
$$
where the first equality follows from the construction of $\xi^{\star}_T(\omega)$, whereas the second equality holds because $V(\fe({\xi_0}))$ is continuous on $\Xi_0$. 
The claim now follows because $\qhat(\omega)$ converges to~${\xi_0}$ for $\mb P_{\xi_0}$-almost every sample~$\omega$.
\end{proof}

\section{Distributionally Robust Offline Policy Optimization}\label{ssec:OPL:dist:shift}

The distributionally robust estimator~$V^\pi_\rho(\qhat)$ analyzed in Section~\ref{ssec:OPE:dist:shift} solves the off-policy evaluation problem in a statistically efficient manner. Thus, it estimates the average reward of a fixed evaluation policy~$\pi$ based on data generated under a behavioral policy~$\pi_0\neq\pi$. We now use~$V^\pi_\rho(\qhat)$ as a building block for an offline policy optimization problem that learns an optimal policy~$\pi$ from data generated under~$\pi_0$. To this end, we define $\Pi_\epsilon=\{\pi\in\Pi: \pi(a|s) \ge \epsilon\;\forall s\in\mc S,\, a\in\mc A\}$ as the family of all exploratory policies that select any action with probability at least $\epsilon> 0$, where $\epsilon$ is sufficiently small to ensure that $\Pi_\epsilon\neq \emptyset$.
We then introduce the distributionally robust policy estimator~$\pi_\rho(\qhat)$ corresponding to a given radius~$\rho\geq 0$, where $\pi_\rho:\Xi\to\Pi_\epsilon$ is any Borel measurable function with
\begin{align}
    \label{def:pistar}
    \pi_\rho(\xi')\in\argmax_{\pi\in\Pi_\epsilon} V^\pi_\rho(\xi') \quad \forall \xi'\in\Xi.
\end{align} 
Restricting~$\Pi$ to~$\Pi_\epsilon$ in~\eqref{def:pistar} simplifies the analysis of~$\pi_\rho(\qhat)$ at the expense of sacrificing flexibility. Note, however, that any policy~$\pi\in\Pi$ gives rise to an $\epsilon$-greedy policy~$\pi_\epsilon\in\Pi_\epsilon$ defined through $\pi_\epsilon(a|s)= \epsilon/A + (1-\epsilon) \pi(a|s)$ for all $a\in\mc A$ and $s\in\mc S$, which approximates~$\pi$ arbitrarily closely.

In the remainder of this section, we investigate the statistical properties of the distributionally robust policy estimator $\pistar(\qhat)$. Specifically, we show that it is efficient in comparison to all estimators $\widehat\pi(\qhat)$ induced by admissible policy functions~$\widehat \pi$ in the sense of the following definition.


\begin{definition}[Admissible policy function]\label{def:pi:admissibility}
    Given~$\rho> 0$ and $\epsilon>0$, a policy function $\pihat:\Xi\to \Pi_\epsilon$ is called admissible if there exist value functions $\widehat V^\pi:\Xi\to\mb R$ parametrized by $\pi\in\Pi_\epsilon$ such that
    \begin{enumerate}[label = (\roman*)]
      \item \label{eq:item:admissible:conti} $\Vhat^{\pi}(\xi')$ is continuous in~$\pi$ and lower semicontinuous in $\xi'$ throughout $\Pi_\epsilon\times\Xi$; 
      \item \label{eq:item:admissible:greedy} $\pihat(\xi')$ is Borel measurable in~$\xi'\in \Xi$, and $\pihat(\xi')\in\arg\max_{\pi\in\Pi_\epsilon}\Vhat^{\pi}(\xi')$ for every~$\xi'\in\Xi_0$;
      \item \label{eq:item:admissible:ucb} 
      the out-of-sample disappointment of the policy $\pihat(\qhat)$ decays exponentially at rate~$\rho$, that is,
      \begin{equation}\label{eq:condition:pi:admissibility}
        \limsup_{T\to\infty} \frac{1}{T} \log\mathbb{P}_{\xi_0}\left( V(f_{\pihat(\qhat)}(\xi_0)) <  \Vhat^{\pihat(\qhat)}(\qhat)\right) \leq -\rho \quad \forall  \xi_0\in\Xi_0,
    \end{equation}
    where~$f_\pi$ is the distribution shift function and $\qhat$ the empirical estimator defined in~\eqref{MDP:estimator}.
    \end{enumerate}
\end{definition}

We say that the admissibility of a policy function $\pihat$ is certified by the corresponding family $\Vhat^{\pi}$, $\pi\in\Pi_\epsilon$, of value functions. We emphasize that the condition of admissibility is very weak. For example, every continuous policy function $\pihat:\Xi\to \Pi_\varepsilon$ is admissible. Indeed, $\pihat(\xi')$ is the unique maximizer of the continuous function $\widehat V^\pi(\xi')= -\|\pi-\pihat(\xi')\|^2_2-C$ across all $\pi\in\Pi_\epsilon$, where $C$ is an arbitrary real constant. To ensure that~\eqref{eq:condition:pi:admissibility} holds, it suffices to make $C$ large. On a related note, the following lemma shows that for every family of value function~$\widehat V^\pi$, $\pi\in\Pi_\epsilon$, satisfying condition~\ref{eq:item:admissible:conti} there exists a policy function~$\widehat\pi$ satisfying condition~\ref{eq:item:admissible:greedy} of Definition~\ref{def:pi:admissibility}. 

\begin{lemma}[Borel measurability of~$\pihat$]\label{lem:pi:quasi:continuous}
If $\epsilon>0$ and $\Vhat^{\pi}(\xi')$ is continuous in~$\pi$ and lower semicontinuous in $\xi'$ throughout $\Pi_\epsilon\times\Xi$, then there is a Borel measurable function $\pihat:\Xi\to\Pi_\epsilon$ with 
\[
    \pihat(\xi')\in\arg\max_{\pi\in\Pi_\epsilon}\Vhat^{\pi}(\xi')\quad \forall \xi'\in\Xi.
\]
\end{lemma}

\begin{proof}
As~$\Pi_\epsilon$ is compact and~$\Vhat^{\pi}(\xi')$ is continuous in~$\pi\in\Pi_\epsilon$ for every fixed~$\xi'\in\Xi$, the optimization problem $\max_{\pi\in\Pi_\epsilon}\Vhat^{\pi}(\xi')$ is solvable. Hence, the solution mapping $\argmax_{\pi\in\Pi_\epsilon} \widehat V^\pi(\xi')$ is non-empty- and closed-valued throughout~$\Xi$. 
Next, note that $\widehat V^\pi(\xi')$ is continuous in $\pi$ and lower semicontinuous and thus Borel measurable in~$\xi'$ throughout $\Pi_\epsilon\times\Xi_0$. In addition, the feasible set~$\Pi_\epsilon$ is constant and thus trivially Borel measurable in~$\xi$ throughout~$\Xi$. By \citep[Examples~14.29 \&~14.32]{rockafellar2009variational}, the function $\psi:\Pi_\epsilon\times\Xi\to [-\infty,\infty)$ defined through $\psi(\pi,\xi')=-\widehat V^\pi(\xi')$ if $\pi\in\Pi_\epsilon$; $=\infty$ otherwise constitutes a normal integrand in the sense of \citep[Definition~14.27]{rockafellar2009variational}. Consequently, the solution mapping $\arg\max_{\pi\in\Pi_\epsilon} \widehat V^\pi(\xi')$ is Borel measurable in~$\xi'$ throughout $\Xi$ by virtue of \citep[Theorem~14.37]{rockafellar2009variational}. By \citep[Corollary~14.6]{rockafellar2009variational}, the solution mapping thus admits a Borel measurable selector.
\end{proof}



Together with Corollary~\ref{cor:admissible:dr-policy}, Lemma~\ref{lem:pi:quasi:continuous} implies that the distributionally robust policy function~$\pi_\rho$ is well-defined, that is, there is indeed a Borel measurable function~$\pi_\rho$ satisfying~\eqref{def:pistar}. In order to show that~$\pi_\rho$ is admissible in the sense of Definition~\ref{def:pi:admissibility}, it thus suffices to verify condition~\ref{eq:item:admissible:ucb}.

\begin{theorem}[Out-of-sample disappointment of $\pistar(\qhat)$]\label{thm:gen:bound:pi}
For every~$\rho>0$, $\epsilon>0$ and $ \xi_0\in\Xi_0$, the distributionally robust policy estimator $\pistar(\qhat)$ satisfies
\begin{align*}
    \limsup_{T\to\infty} \frac{1}{T} \log\mathbb{P}_{\xi_0}\left( V(\fpistar(\xi_0)) <  V^{\pistar(\qhat)}_\rho(\qhat)\right) \leq -\rho.
\end{align*}
\end{theorem}

Theorem~\ref{thm:gen:bound:pi} is a direct consequence of Theorem~\ref{thm:pi:finite:sample} below. Thus, its proof is omitted.

\begin{theorem}[Finite-sample guarantee for $\pistar(\qhat)$]\label{thm:pi:finite:sample}
For every $\xi_0\in\Xi_0$, there is $\bar c>0$ such that the distributionally robust estimator $\pistar(\qhat)$ satisfies the following for all $\rho>0$, $\epsilon>0$ and~$T\in\mathbb N$.
\begin{align*}
    \frac{1}{T} \log\mathbb{P}_{\xi_0}\left( V(\fpistar(\xi_0)) <  V^{\pistar(\qhat)}_\rho(\qhat)\right) \leq \frac{1}{T}(\log (T)+\bar c+ S^2A\log (T+1) )-\rho
\end{align*}
\end{theorem}
\begin{proof}
Define the disappointment set $\mc D(\pi)=\{\xi'\in\Xi: V(\fpi(\xi_0)) <  V^\pi_\rho(\xi')\}$ corresponding to a fixed policy $\pi\in\Pi_\epsilon$, which is open because $V^\pi_\rho(\xi')$ is lower semicontinuous by virtue of Lemma~\ref{lem:lsc}. Next, observe that, for every fixed estimator realization $\xi'\in\Xi$, the following implications hold.
\begin{align*}
    V(f_{\pistar(\xi')}(\xi_0)) <  V^{\pistar(\xi')}_\rho(\xi')&\implies
    \exists \pi\in\Pi_\epsilon \text{ with } V(\fpi(\xi_0))<V^\pi_\rho(\xi')\\
    &\implies \xi'\in\cup_{\pi\in\Pi_\epsilon} \mc D(\pi) 
\end{align*}
This in turn implies that
\begin{align*}
 \frac{1}{T}\log  \mathbb{P}_{ \xi_0}& \left( V(\fpistar(\xi_0)) <  V^{\pistar(\qhat)}_\rho(\qhat)\right) \le  \frac{1}{T}\log \mb P_{ \xi_0}\left(\qhat\in\cup_{\pi\in\Pi_\epsilon} \mc D(\pi)\right)\\
 &\leq \frac{1}{T}(\log T+\bar c+ d^2\log (T+1) )-\inf_{\xi'\in \cup_{\pi\in\Pi_\epsilon}\mc D(\pi)}{\Lc{\xi'}{\xi}} \\
 &\leq \frac{1}{T}((SA)^2\log (T+1)+\log (T)+\bar c  )-\rho \quad   \forall T\in\mb N,
\end{align*}
where the second inequality follows from Corollary~\ref{ldp:finite:q}. The third equality can be justified as follows. The definition of~$V^{\pi}_\rho$ in~\eqref{def:vstar:q} readily implies that $V(\fpi(\xi_0))\ge V^\pi_\rho(\xi')$ whenever $\Lc{\xi'}{ \xi_0}\le\rho.$ By contraposition, this means that $\Lc{\xi'}{ \xi_0}>\rho$ whenever $V(\fpi(\xi_0))< V^\pi_\rho(\xi')$ for some $\pi\in\Pi$. Hence, the infimum in the second line of the above expression is bounded below by~$\rho$. 
\end{proof}

We are now ready to show that the distributionally robust policy estimator $\pi_\rho(\qhat)$ represents the least conservative admissible policy estimator in a sense made precise in the following theorem.

\begin{theorem}[Statistical efficiency of $\pi_\rho(\qhat)$]\label{thm:pareto:V:pistar}
    For every fixed $\rho>0$ and $\epsilon>0$, and for any admissible policy function $\widehat\pi$ and the corresponding family of value functions $\Vhat^{\pi}$, $\pi\in\Pi_\epsilon$, we have
    \begin{align}\label{ineq:pareto:pi}
        V^{\pistar(\xi')}_\rho(\xi') \geq \Vhat^{\pihat(\xi')}(\xi')\quad \forall \xi'\in\Xi.
    \end{align}
\end{theorem}

\begin{proof}
Select any~$\xi'\in\Xi$. 
We have
\begin{equation*}
\begin{aligned}
   \widehat V^{\pihat(\xi')}(\xi') \le  V_\rho^{\pihat(\xi')}(\xi') \le V_\rho^{\pi_\rho(\xi')}(\xi'),
\end{aligned}
\end{equation*}
where the first inequality follows from Theorem~\ref{thm:pareto:vstar} for $\pi= \pihat(\xi')$, and the second inequality exploits the definition of~$\pi_\rho$ in~\eqref{def:pistar}. Thus, the claim follows.
\end{proof} 

By condition~\ref{eq:item:admissible:greedy} of Definition~\ref{def:pi:admissibility}, at $\xi'=\qhat$ the right hand side of~\eqref{ineq:pareto:pi} equals $\max_{\pi\in\Pi_\epsilon} \widehat V^\pi(\qhat)$ and can thus be viewed as a generic estimator for the optimal value of the offline policy optimization problem. Condition~\ref{eq:item:admissible:ucb} of Definition~\ref{def:pi:admissibility} ensures that this estimator is conservative, that is, it overestimates the achievable expected reward of the policy $\pihat(\qhat)$ only with a small probability of at most $e^{-\rho T+o(T)}$. Thus, it provides a lower confidence bound on the optimal expected reward. Similarly, by~\eqref{def:pistar}, at $\xi'=\qhat$ the left hand side of~\eqref{ineq:pareto:pi} equals $\max_{\pi\in\Pi_\epsilon} V_\rho^\pi(\qhat)$, and Theorem~\ref{thm:gen:bound:pi} ensures that it provides a lower confidence bound with the same significance level. Theorem~\ref{thm:pareto:V:pistar} thus asserts that the distributionally robust policy estimator~$\pi_\rho(\qhat)$ provides the least conservative lower confidence bound among all policy estimators with the same significance level $e^{-\rho T+o(T)}$.
\section{Numerical Solution Schemes}
\label{sec:computational:aspects}

The distributionally robust value and policy estimators introduced in Sections~\ref{ssec:OPE:dist:shift} and~\ref{ssec:OPL:dist:shift} are statistically efficient. It remains to be discussed how these estimators can be computed efficiently. 
To this end, we will adapt the actor-critic algorithm by \citet[\S~4]{li2023policy}, which was originally developed for robust MDPs with discounted cost criteria, to robust MDPs with average reward criteria.

\subsection{Reparametrization}\label{ssec:reparametrization}
As a preparation, we first show that the robust policy evaluation problem in~\eqref{def:vstar:q} can be reformulated as an optimization problem over transition kernels with a non-convex objective function and a convex feasible set. Thus, for any state-action-next-state distributions $\xi', \xi \in \Xi_0$ we define the corresponding transition kernels $Q', Q\in\mc Q_0$ and policies $\pi',\pi\in\Pi_0$ through~\eqref{eq:q-from-xi} and~\eqref{eq:pi-from-xi}, respectively. In addition, we let $\mu',\mu\in\Delta(\mc S\times\mc A)$ be the corresponding stationary state-action distributions. The conditional relative entropy~\eqref{def:Dm:Theta} between $\xi$ and $\xi'$ can then be equivalently expressed as
\begin{equation}\label{Dm:alt:expression}
\begin{split}
    \Lc{\xi'}{\xi} &= \sum_{s\in\mc S } \sum_{a\in\mc A}\mu'(s,a)\left(\log \frac{\mu'(s,a)}{\sum_{\tilde a\in\mc A}\mu'(s,\tilde a)}-\log \frac{\mu(s,a)}{\sum_{\tilde a\in\mc A}\mu(s,\tilde a)}\right) \\
    &\qquad +\sum_{s,s'\in\mc S} \sum_{a,a'\in\mc A}\xi'(s,a,s')\left(\log \frac{\xi'(s,a,s')}{\mu'(s,a)}-\log \frac{\xi(s,a,s')}{\mu(s,a)}\right) \\
    &=\sum_{s\in\mc S} \sum_{a\in\mc A}\mu'(s,a) \mathsf{D}(\pi'(\cdot |s)\|\pi(\cdot |s))+\sum_{s\in\mc S} \sum_{a\in\mc A}\mu'(s,a)\mathsf{D}(Q'(\cdot|s,a)\|Q(\cdot|s,a)),
\end{split}
\end{equation}
where the second equality follows from~\eqref{eq:q-from-xi} and~\eqref{eq:pi-from-xi}. We can use~\eqref{Dm:alt:expression} to recast the distributionally robust value function~\eqref{def:vstar:q} as the optimal value of a parametric minimization problem over~$Q$. 

\begin{lemma}[Reformulation of $V^\pi_\rho(\xi')$]\label{lemma:vstar:reformulation:Q}
For any fixed $\pi\in\Pi_0$, $\rho\geq 0$ and $\xi'\in\Xi_0$, we have
\begin{align}\label{vstar:Q:formulation}
    V^{\pi}_\rho(\xi')=\min_{Q\in \mc Q_0}\left\{\sum_{x\in\mc X}r(x)\mu_{\pie,Q}(x) : \sum_{x\in\mc X}\mu'(x)\mathsf{D}(Q'(\cdot|x)\|Q(\cdot|x))\le \rho\right\},
\end{align}
where $Q'$ is the transition kernel induced by~$\xi'$ via~\eqref{eq:q-from-xi} and $\mu_{\pie,Q}$ is the unique positive solution of the equations $\sum_{x\in\mc X} \mu_{\pi,Q}(x)=1$ and $\mu_{\pi,Q}(x)=\sum_{y\in\mc X}\mu_{\pi,Q}(y) P(y,x)$ with $P$ defined as in~\eqref{expr:P:from:pi:Q}.
\end{lemma}

The stationary distribution $\mu_{\pi,Q}$ is well-defined thanks to the Perron-Frobenius theorem. Note that the minimization problem~\eqref{vstar:Q:formulation} is independent of the policy~$\pi'$ corresponding to~$\xi'$. In the remainder of the paper we use $\mc Q_\rho(\xi')$ as a shorthand for the feasible set of~\eqref{vstar:Q:formulation}.

\begin{proof}{\textbf{of Lemma~\ref{lemma:vstar:reformulation:Q}}}
Recall that $V^{\pi}_\rho(\xi')$ is defined as the optimal value of problem~\eqref{def:vstar:q}. In the first part of the proof we show that the optimal value of~\eqref{def:vstar:q} is at least as large as that of~\eqref{vstar:Q:formulation}. To this end, suppose that $\xi_0\in\Xi_0$ is a feasible solution of~\eqref{def:vstar:q}, and define $\xi=f_\pi(\xi_0)$, where $f_\pi$ is as in Definition~\ref{def:f}. Then, the transition kernel $Q\in\mc Q_0$ induced by~$\xi_0$ through~\eqref{eq:q-from-xi} satisfies 
\[
    \sum_{x\in\mc X}r(x)\mu_{\pie,Q}(x) =\sum_{s,s'\in\mc S}\sum_{a\in\mc A}r(s,a)\xi(s,a,s')= V(\fe(\xi_0)),
\] 
where the two equalities follow from~\eqref{xi:alt:pi:Q} and from the definition of~$V$ at the beginning of Section~\ref{ssec:OPE:dist:shift}, respectively. Thus, the objective function value of~$Q$ in~\eqref{vstar:Q:formulation} equals that of~$\xi_0$ in~\eqref{def:vstar:q}. Next, let~$\pi_0\in\Pi_0$ be the policy induced by~$\xi_0$ through~\eqref{eq:pi-from-xi}. By construction, we also have
\begin{align*}
   \sum_{x\in\mc X}\mu'(x)\mathsf{D}(Q'(\cdot|x)\|Q(\cdot|x))
    &= \Lc{\xi'}{\xi_0} - \sum_{s\in\mc S}\sum_{a\in\mc A} \mu'(s,a) \mathsf{D}(\pi'(\cdot |s)\|\pi_0(\cdot |s))\le \rho,
\end{align*} 
where the equality follows from~\eqref{Dm:alt:expression}, while the inequality holds because~$\xi_0$ is feasible in~\eqref{def:vstar:q} such that $\Lc{\xi'}{\xi_0}\le\rho$ and because the relative entropy is non-negative. Hence, $Q$ is feasible in~\eqref{vstar:Q:formulation}. In summary, these insights confirm that the optimal value of~\eqref{def:vstar:q} is at least as large as that of~\eqref{vstar:Q:formulation}.

In the second part of the proof we show that the optimal value of~\eqref{def:vstar:q} is at most as large as that of~\eqref{vstar:Q:formulation}. To this end, suppose that $Q\in\mc Q_0$ is feasible in~\eqref{vstar:Q:formulation}. Defining $\xi_0\in\Xi_0$ as the stationary state-action-next-state distribution corresponding to~$\pi'$ and~$Q$ and setting $\xi=f_\pi(\xi_0)$, we then find 
\[
    V(\fe(\xi_0))=\sum_{s,s'\in\mc S}\sum_{a\in\mc A}r(s,a)\xi(s,a,s')=\sum_{x\in\mc X}r(x)\mu_{\pi,Q}(x),
\] 
where the equalities follow again from the definition of~$V$ and from~\eqref{xi:alt:pi:Q}, respectively. Thus, the objective function value of~$\xi_0$ in~\eqref{def:vstar:q} equals that of~$Q$ in~\eqref{vstar:Q:formulation}.
By construction of~$\xi_0$, we also have
\begin{align*}
 \Lc{\xi'}{\xi_0}&=\sum_{s\in\mc S}\sum_{a\in\mc A}\mu'(s,a) \mathsf{D}(\pi'(\cdot |s)\|\pi'(\cdot |s))+\sum_{x\in\mc X}\mu'(x) \mathsf{D}(Q'(\cdot|x)\|Q(\cdot|x))
 \\&=\sum_{x\in\mc X}\mu'(x)\mathsf{D}(Q'(\cdot|x)\|Q(\cdot|x))\le\rho,
\end{align*}
where the two equalities follow from~\eqref{Dm:alt:expression} and from the trivial relation $\mathsf{D}(\pi'(\cdot |s)\|\pi'(\cdot |s))=0$, respectively, whereas the inequality holds because~$Q$ is feasible in~\eqref{vstar:Q:formulation}. Hence, $\xi$ is feasible in~\eqref{def:vstar:q}.
In summary, these insights confirm that the optimal value of~\eqref{def:vstar:q} is at most as large as that of~\eqref{vstar:Q:formulation}. 

The above reasoning implies that the optimal values of~\eqref{def:vstar:q} and~\eqref{vstar:Q:formulation} match. Recalling that the minimum of~\eqref{def:vstar:q} is attained because~$\xi'\in\Xi_0$, the first part of the proof thus implies that the minimum of problem~\eqref{vstar:Q:formulation} is attained, too. This observation completes the proof.
\end{proof}

Note that the feasible region of problem~\eqref{vstar:Q:formulation} is convex because~$\mc Q_0$ is a convex set and the relative entropy is a convex function. However, the objective function of~\eqref{vstar:Q:formulation} is generically non-convex because the stationary distribution $\mu_{\pi,Q}$ fails to be convex in~$Q$ \citep[Remark 13]{li2021distributionally}. Note also that problem~\eqref{vstar:Q:formulation} evaluates the worst-case average reward of the policy~$\pi$ across all transition kernels~$Q$ in the non-rectangular uncertainty set~$\mc Q_\rho(\xi')$ {\em without} a distribution shift.

\begin{example}[Non-rectangularity of $\mc Q_\rho(\xi')$]\label{rmk:non-rectangularity}
Assume that $\mc S = \{s_1, s_2\}$ and $\mc A=\{a\}$, and fix an arbitrary \(\xi' \in \Xi_0\). In this case, the uncertainty set~$\mc Q_\rho(\xi')$ simplifies to
\begin{equation*}
    \mc Q_\rho(\xi') = \left\{Q\in \mc Q_0 :  \mu'_{\mc S}(s_1)\mathsf{D}(Q'(\cdot|s_1,a)\|Q(\cdot|s_1,a))+\mu'_{\mc S}(s_2)\mathsf{D}(Q'(\cdot|s_2,a)\|Q(\cdot|s_2,a))\le \rho \right\}.
\end{equation*}
Thus, any transition kernel \(Q \in \mathcal{Q}_\rho(\xi')\) must respect the following inequality:
\[\mu'_{\mc S}(s_1)\mathsf{D}(Q'(\cdot|s_1,a)\|Q(\cdot|s_1,a)) \le \rho - \mu'_{\mc S}(s_2)\mathsf{D}(Q'(\cdot|s_2,a)\|Q(\cdot|s_2,a)).\]
This inequality demonstrates that the permissible next-state distribution of $Q(\cdot|s_1,a)$ for state \(s_1\) and action $a$ is contingent upon $Q(\cdot|s_2,a)$ for state \(s_2\) and action $a$. Such interdependence directly contradicts the definition of \(s\)-rectangularity \citep[\S~2.1]{wiesemann2013robust}, which requires that the ambiguity set decompose into independent, state-specific components.
\end{example}

Robust policy evaluation problems (and thus robust MDPs) with non-rectangular uncertainty sets are generically intractable. Theorem~1 in \citep{wiesemann2013robust} shows that when the uncertainty set of the transition kernel \(Q\) is a convex polytope that does not satisfy any rectangularity conditions---specifically, it is neither \((s,a)\)-rectangular, \(s\)-rectangular, nor \(r\)-rectangular---then the robust policy evaluation problem can be reduced to an integer feasibility problem and thus is NP-hard.

\subsection{Actor-Critic Algorithm}
Given an oracle that outputs approximate solutions for the robust policy evaluation problem in~\eqref{vstar:Q:formulation}, the robust policy optimization problem~\eqref{def:pistar} can be addressed with a variant of the actor-critic algorithm developed by \citet{li2023policy} for robust MDPs with a {\em discounted} cost objective; see Algorithm~\ref{alg:PG-min-oracle}. Throughout this section we use $V^\pi_Q=\sum_{x\in\mc X} r(x)\mu_{\pi,Q}(x) $ as shorthand for the long-run average reward of the policy~$\pi\in\Pi_\epsilon$ under the transition kernel~$Q\in\mc Q_0$. 

\begin{algorithm}[h!] 
  \caption{Actor-critic algorithm for solving the robust policy optimization problem~\eqref{def:pistar}}
  \label{alg:PG-min-oracle}
\begin{algorithmic}[1]
\REQUIRE Iteration number $K$, step size $\eta>0$, tolerance $ \delta>0$
\STATE Initialize  $\pi^{(0)}\in \Pi_\epsilon$, $k=0$
\WHILE{$k\le K$}
  \STATE \textit{Critic}: Find $\Pk\in\mc Q_\rho(\xi')$ such that
  $\Value{\Pk}{\pik}\le \Value{\rho}{\pik}(\xi')+\delta$
  \label{alg:PG-min-oracle:line:PE}
  \STATE \textit{Actor}: $\pikp=\mathrm{Proj}_{\Pi_\epsilon} \left(\pik+\eta\nabla_{\pi} \Value{\Pk}{\pik}\right)$
  \label{alg:PG-min-oracle:line:MD}
  \STATE $k \rightarrow k+1$
\ENDWHILE
\end{algorithmic}
\end{algorithm}

In each iteration~$k$, Algorithm~\ref{alg:PG-min-oracle} first computes a $\delta$-optimal solution $\Pk$ of the robust policy evaluation problem~\eqref{vstar:Q:formulation} associated with the current policy~$\pik$ (\textit{critic}) and then applies a projected gradient step to find a new policy~$\pikp$ that locally improves the value function associated with the current transition kernel $\Pk$ (\textit{actor}). The critic's subproblem can be solved with Algorithm~\ref{alg:PLD} described in Section~\ref{ssec:critic} below, for example, which outputs a $\delta$-optimal solution of the robust policy evaluation problem with high probability. The actor's subproblem simply consists in computing a policy gradient and projecting a vector onto the simplex~$\Pi_\epsilon$ as described in Section~\ref{ssec:actor} below.

\subsubsection{Actor}\label{ssec:actor}

Projecting a vector onto $\Pi_\epsilon$ is a standard operation that admits highly efficient implementations; see, {\em e.g.},  \citep{wang2013projection}. Thus, the main computational burden of the actor's subproblem is associated with the computation of the policy gradient~$\nabla_\pi\Value{Q}{\pi}$. In order to derive a concise formula for the policy gradient, it is useful to introduce a differential action-value function.

\begin{definition}[Differential action-value function]
\label{def:diff-action-value-function}
For any fixed policy $\pi\in\Pi_\epsilon$ and transition kernel $Q\in\mc Q_0$, the 
differential action-value function $\bias{\pi}{Q}:\mc S\times \mc A\to\mb R $ is defined through
\begin{align*}
    \bias{\pi}{Q}(s,a)=\lim_{T\to\infty}\frac{1}{T}\sum_{t=1}^T\mathbb{E}_{\mb P_{\pi,Q}}\left[ \left. \sum_{\tau=1}^{t}\left(r(s_\tau, a_\tau)-V^\pi_Q\right) \,\right|\, s_1=s, \, a_1=a\right].
\end{align*}
\end{definition}
We emphasize that the limit in Definition~\ref{def:diff-action-value-function} exists and is finite \citep[Section~8.2.1]{puterman2005markov}. Note also that, strictly speaking, the gradient $\nabla_\pi V^\pi_Q$ fails to exist because the function $V^\pi_Q$ was only defined on~$\Pi_0\times\mc Q_0$ and because the interior of~$\Pi_0$ is empty. Using \citep[Lemma~4]{li2021distributionally}, however, one can show that $\Value{Q}{\pi}$ is given by a rational (and thus analytic) function of the matrix~$P$ defined as in~\eqref{expr:P:from:pi:Q}. As~$P$ is linear in~$\pi$, $\Value{Q}{\pi}$ thus constitutes a rational function of~$\pi$. One can also show that the denominator of this rational function is bounded away from zero on a neighborhood of~$\Pi_\epsilon$ for any~$\epsilon>0$. In the following, we interpret $\nabla_\pi V^\pi_Q$ as the gradient of this analytic extension of~$V^\pi_Q$, which is well-defined on~$\Pi_\epsilon$ because the interior of any neighborhood of~$\Pi_\epsilon$ covers~$\Pi_\epsilon$.


\begin{lemma}[Policy gradient {\citep[Lemma~7]{li2022stochastic}}] \label{lemma:policy:gradient:OPL}
If $Q\in \mc Q_0$ and $\pi\in\Pi_\epsilon$, then 
$$
    \frac{\partial \Value{Q}{\pi}}{\partial\pi(a|s)}=\sum_{\tilde a\in\mc A} \mu_{\pi,Q}(s,\tilde a)\bias{\pi}{Q}(s, a) \quad\forall a\in\mc A, s\in\mc S,
$$
where $\mu_{\pie,Q}$ is defined as in Lemma~\ref{lemma:vstar:reformulation:Q}.
\end{lemma}

Note that~$\nabla_\pi V^\pi_Q$ is Lipschitz continuous in~$\pi$ throughout~$\Pi_\epsilon$ thanks to \citep[Lemma~17]{li2021distributionally}. As~$\Pi_\epsilon$ is compact, this readily implies that~$V^\pi_Q$ is also Lipschitz continuous in~$\pi$ throughout~$\Pi_\epsilon$. By using a variant of  {\citep[Theorem~4.5]{li2023policy}}, one can thus show that Algorithm~\ref{alg:PG-min-oracle} converges to a global maximizer of the robust offline policy optimization problem in~\eqref{def:pistar}.

\begin{theorem}[Convergence of Algorithm~\ref{alg:PG-min-oracle}]\label{thm:PO:convergence}
Assume that $\Value{Q}{\pi}$ is $L$-Lipschitz and $\nabla_{\pi} \Value{Q}{\pi}$ is $\ell$-Lipschitz in~$\pi\in\Pi_\epsilon$ uniformly for all $Q\in \mathcal{Q}_\rho(\xi')$, and set the distribution mismatch coefficient~to
\[
    C=\max _{\pi, \pi^{\prime} \in \Pi_\epsilon} \max_{Q \in \mathcal{Q}_\rho(\xi')} \max_{s \in \mathcal{S}} \frac{\mu_{\pi,Q}(s)}{\mu_{\pi',Q}(s )}.
\] 
If $\delta=\frac{L}{2}\sqrt{2S/K}$ and $\eta=\frac{1}{L}\sqrt{2S/K}$, then the iterates $\pik$ of Algorithm~\ref{alg:PG-min-oracle} satisfy
\begin{align*}
\frac{1}{K}\sum_{k=0}^{K-1}\left( \Value{\rho}{\pik}(\xi')-\min_{\pi\in\Pi_\epsilon}\Value{\rho}{\pi}(\xi')\right)\le  \frac{(72S)^{1/4}(C\sqrt{2\ell L S}+\frac{L}{2}\sqrt{L/\ell})}{K^{1/4}}. 
\end{align*}
\end{theorem}

Theorem~\ref{thm:PO:convergence} shows that computing a 
$\delta$-optimal solution for the robust offline policy optimization problem in~\eqref{def:pistar} requires at most~$K=\mc O(\delta^{-4})$ iterations. We point out that the convergence rate of Algorithm~\ref{alg:PG-min-oracle} is of the order $\mc O(K^{-1/4})$ whenever~$\eta=\mc O(K^{-1/2})$, and thus it is not necessary to know the Lipschitz constant~$L$ in practice. Setting $\eta=\frac{1}{L}\sqrt{2S/K}$ leads to theoretically optimal constants.


\begin{proof}{\textbf{of Theorem~\ref{thm:PO:convergence}}}
Note that $\mu_{\pi,Q}>0$ for every $\pi\in\Pi_\epsilon$ and $Q\in\mc Q_{\rho}(\xi')\subseteq\mc Q_0$ thanks to Lemma~\ref{lem:irreducibleity:pi:Q:X_t}. In addition, $\mu_{\pi,Q}$ is jointly continuous in~$\pi$ and~$Q$ by \citep[Lemma~17]{li2021distributionally}. Weierstrass' extreme value theorem thus implies that the maxima in the definition of~$C$ are attained such that $C$ is finite and strictly positive. Next, define the Moreau envelope $\Phi_\gamma: \mb R^{S\times A}\to\mb R$ of the distributionally robust value function $V_\rho^\pi$ corresponding to the smoothness parameter~$\gamma > 0$ by
\[
    \Phi_\gamma(\pi)=\min_{\pi'\in\Pi_\epsilon}V_{\rho}^{\pi'}(\xi')+\frac{1}{2\gamma}\|\pi'-\pi\|^2_{\mathbf F},
\]
where $\|\cdot\|_{\mathbf F}$ stands for the Frobenius norm. We then have
\begin{align*}
\frac{1}{K}\sum_{k=0}^{K-1}\left( \Value{\rho}{\pik}(\xi')-\min_{\pi\in\Pi_\epsilon}\Value{\rho}{\pi}(\xi')\right)
&\le \frac{(C\sqrt{2S}+L/(2\ell))}{K}\sum_{k=0}^{K-1}\|\nabla \Phi_{1 / (2 \ell)}(\pik)\|_{\mathbf F}
\\&\le\frac{(C\sqrt{2S}+L/(2\ell))}{\sqrt{K}} \sqrt{\sum_{k=0}^{K-1}\|\nabla \Phi_{1 / (2 \ell)}(\pik)\|_{\mathbf F}^2}
\\&\le \frac{(C\sqrt{2S}+L/(2\ell))(72S)^{1/4}(\ell L)^{1/2}}{K^{1/4}},
\end{align*}
where the first inequality follows from~\citep[Lemma~4.4]{li2023policy}, which applies because~$\Value{Q}{\pi}$ is $L$-Lipschitz and $\nabla_{\pi} \Value{Q}{\pi}$ is $\ell$-Lipschitz in~$\pi\in\Pi_\epsilon$ uniformly for all $Q\in \mathcal{Q}_\rho(\xi')$. The second and the third inequalities exploit Jensen's inequality and~\citep[Lemma~4.3]{li2023policy}, respectively.
\end{proof}

\subsubsection{Critic} \label{ssec:critic}
We now show that the critic's robust policy evaluation problem~\eqref{vstar:Q:formulation} can be solved approximately with a randomized policy gradient method that offers global convergence guarantees; see Algorithm~\ref{alg:PLD}. 

Throughout this section 
we assume that the uncertainty set admits a reparametrization of the form $\mc Q_\rho(\xi')=\{Q^\lambda : \lambda\in\Lambda\}$, where $\Lambda \subseteq \mb R^q$ is a solid parameter set, and~$Q^\lambda$ is an affine function. As~$\Lambda$ is solid, its linear span coincides with the ambient space~$\mb R^q$. A reparametrization of the uncertainty set with these properties and~$q=A(S-1)$ exists; see \citep[\S~5]{wiesemann2013robust}.

\begin{algorithm}[ht!]     \caption{Projected Langevin dynamics for solving the robust policy evaluation problem~\eqref{vstar:Q:formulation}}
  \begin{algorithmic}[1]
  \REQUIRE Iteration number $M\in\mb N$, step size $\eta >0$, Gibbs parameter $\beta>1$
  \STATE Initialize  $\lambda^{(0)}\in\Lambda$, $ m=0$
  \WHILE{$m\le M-1$}
  \STATE Sample $w_{m+1}\sim \mathcal{N}\left(0, I_q\right)$ 
    \STATE Find $\Ptp=\mathrm{Proj}_{\Lambda}\left(\Pt-\eta \left. \nabla_{\lambda} \Value{Q^{\lambda}}{\pi}\right|_{\lambda=\Pt} +\sqrt{2 \eta/\beta} w_{m+1}\right)$ \label{step:pld}
    \STATE $m \rightarrow m+1$
\ENDWHILE
\end{algorithmic}
\label{alg:PLD}
\end{algorithm}

In each iteration Algorithm~\ref{alg:PLD} applies a projected stochastic gradient step. The projection onto the convex set~$\Lambda$ is a standard operation that admits efficient implementations \citep{usmanova2021fast}. We now address the computation of the adversary's policy gradient. To this end, note that the gradient $\nabla_Q V^\pi_Q$ fails to exist because $V^\pi_Q$ was only defined on~$\Pi_0\times\mc Q_0$ and because the interior of~$\mc Q_0$ is empty. Using a similar reasoning as in Section~\ref{ssec:actor} and recalling that $\xi'\in\Xi_0$, however, $\Value{Q}{\pi}$ can be extended to an analytic function of~$Q$ on a neighborhood of~$\mc Q_\rho(\xi')$. In the following, we interpret $\nabla_Q V^\pi_Q$ as the gradient of this analytic extension of~$V^\pi_Q$. In order to derive a concise formula for the adversary's policy gradient, we introduce a differential action-next-state value function.
\begin{definition}[Differential action-next-state value function]
\label{def:diff-action-next-state}
For any fixed policy $\pi\in\Pi_\epsilon$ and transition kernel $Q\in\mc Q_0$, the 
differential action-next-state value function $J^\pi_Q:\mc S\times \mc A \times \mc S \to\mb R $ 
is defined through
\begin{align*}
J^{\pi}_{Q}(s,a,s')=\lim_{T\to\infty}\frac{1}{T}\sum_{t=1}^T\mathbb{E}_{\mb P_{\pi,Q}}\left[\sum_{\tau=1}^{t}\left(r(s_\tau, a_\tau)-V^\pi_Q\right) | s_1=s, a_1=a, s_2=s'\right].
\end{align*}
\end{definition}

One can show that the limit in Definition~\ref{def:diff-action-next-state} exists and is finite \citep[Section~8.2.1]{puterman2005markov}. The next lemma provides a formula for the adversary's policy gradient. Its proof widely parallels that of~\citep[Lemma 1]{li2023policy} and is thus omitted.

\begin{lemma}[Adversary's policy gradient] \label{lemma:policy:gradient:OPE}
For any $\pi\in\Pi_0$ and $\lambda\in\Lambda$, we have 
$$
\nabla_\lambda \Value{Q^\lambda}{\pi}=\sum_{s,s'\in\mc S}\sum_{a\in\mc A} \mu_{\pi,Q^\lambda}(s,a)J^\pi_{Q^\lambda}(s, a,s')\nabla_\lambda Q^\lambda(s'|s,a).$$
\end{lemma}

By adapting {\citep[Theorem~3]{li2023policy}} to our setting, we can now show that Algorithm~\ref{alg:PLD} converges in expectation to a global minimizer of the robust policy evaluation problem in~\eqref{def:vstar:q}. 

\begin{theorem}[Convergence of Algorithm~\ref{alg:PLD}]
\label{thm:convergence:PLD}
    If $\delta>0$, $\eta<1/2$, $\pi\in\Pi_0$, and $\gamma\in (0,1)$, there exist universal constants $a>4$, $b>1$ and $c_1,c_2,c_3>0$ such that for all $\beta \geq c_1^{-1}\left(2 q/(c_1(1-\gamma) \delta e)\right)^{1 / \gamma}$ and $M\ge \max\{4,c_2\exp (c_3 q^b)/\delta^a\}$ the distribution~$\nu_M$ of the output  $\lambda^{(M)} $ of Algorithm~\ref{alg:PLD} satisfies
$$\mb E_{\lambda\sim\nu_M} [\Value{Q^\lambda}{\pi}]\le V_\rho^\pi(\xi')+\delta.$$
\end{theorem}
\begin{proof}
Recall that $\Lambda$ is a solid convex body. In addition, note that $\nabla_Q\Value{Q}{\pi}$ is Lipschitz continuous in~$Q\in \mc Q_\rho(\xi')$ thanks to \citep[Lemma~17]{li2021distributionally}. As $Q_\lambda$ is affine in $\lambda$ and $\mc Q_\rho(\xi')=\{Q^\lambda : \lambda\in\Lambda\}$, we may then use the chain rule to deduce that $\nabla_\lambda\Value{Q^\lambda}{\pi}$ is Lipschitz continuous in~$\lambda\in\Lambda$. Thus, the claim follows directly from \citep[Theorem 3]{li2023policy}.
\end{proof}

Theorem~\ref{thm:convergence:PLD} implies that the number of iterations~$M$ required by Algorithm~\ref{alg:PLD} to compute a $\delta$-optimal solution for the robust policy evaluation problem~\eqref{def:vstar:q} grows exponentially with both the dimension~$q$ of the parameter~$\lambda$ and the number of accuracy digits~$\log(1/\delta)$. This curse of dimensionality is expected in view of the hardness result by~\citet[Theorem~1]{wiesemann2013robust}. However, Algorithm~\ref{alg:PLD} is conceptually simple and guarantees convergence to a global minimum, even though the uncertainty set fails to be rectangular. Moreover, by leveraging Markov's inequality, we can transform the convergence-in-expectation result from Theorem~\ref{thm:convergence:PLD} into a probabilistic bound.

\begin{corollary}[Probabilistic suboptimality guarantee]
If all assumptions of Theorem~\ref{thm:convergence:PLD} hold, then we have $\mb P_{\lambda\sim\nu_M}[ \Value{Q^{\lambda}}{
\pi}>V^\pi_\rho(\xi') - \delta/\gamma]\ge 1-\gamma$ for all $\gamma\in(0,1)$.
\end{corollary}

\section{Numerical Results}\label{sec:numerical:experiments}

We now assess the out-of-sample properties of the proposed distributionally robust value and policy estimators in two numerical experiments. The first experiment revolves around a stochastic GridWorld system and tests the distributionally robust value estimator from Section~\ref{ssec:OPE:dist:shift}. The second experiment tests the distributionally robust policy estimator from Section~\ref{ssec:OPL:dist:shift} in the context of a standard machine replacement problem. All experiments are implemented in Python, and the code for reproducing all numerical results is available from \url{https://github.com/mengmenglior/offline-rl}.

\subsection{Off-Policy Evaluation: GridWorld System}

The first experiment is built around a GridWorld problem ubiquitous in reinforcement learning \citep{sutton2018reinforcement}. The state space $\mc S$ consists of the $25$ cells of a $5\times 5$ grid, and the action space~$\mc A=\{0,1,2,3\}$ includes the $4$ directions ``up,'' ``down,'' ``left'' and ``right.'' An agent aims to reach the Goal State in cell~$1$ (top left) while avoiding the Bad State in cell~$25$ (bottom right). Selecting action $a\in\mc A$ in state $s\in\mc S$ moves the agent to $s'\in\mc S$ with probability $Q(s'|s,a)$. The agent receives a reward of $0$ in the Goal State, $-5$ in the Bad State, and $-1.5$ elsewhere. The initial state~$s_0$ follows the uniform distribution on~$\mc S$, which we denote by~$\rho$. The transition probabilities are defined as follows. Let $\mc S(s)\subseteq \mc S$ be the set of cells adjacent to~$s$. If~$s'$ is the adjacent cell in direction~$a$, then $Q(s'|s,a)=0.7$; if $s'$ is any other adjacent cell, $Q(s'|s,a)=0.1$; if $s'=s$, $Q(s'|s,a)=1-\sum_{s''\in \mc S(s)}Q(s''|s,a)$; otherwise, $Q(s'|s,a)=0$. These rules apply even if no adjacent cell exists in direction~$a$. The agent observes a state-action trajectory $\{ (S_t,A_t)\}_{t=1}^T$ generated under the behavioral policy $\pib \in \Pi_0$, which selects each action $a\in\mc A$ with probability $0.9\times 0.1^a/(1-0.1^4)$, and leverages this data to estimate the average reward~$V(f_\pi(\xi_0))$ of the evaluation policy $\pie\in\Pi_0$, which selects each action $a\in\mathcal A$ with probability~$A^{-1}$. Thus, $\pi_0$ resembles a geometric distribution on the action space, and the density ratio $\pi(a|s)/\pi_0(a|s)$ increases rapidly with~$a$, which poses a well-known challenge in off-policy evaluation literature~\citep{liu2018breaking}.

We compare our distributionally robust value estimator $V_\rho^\pi(\qhat)$ against the marginalized importance sampling estimator~$\widehat V^\pi_{\rm MIS}(\qhat)$ by~\citet{liu2018breaking}, which is known to have low variance, and against the distributionally robust estimator~$\widehat V^\pi_{\rm OT}(\qhat)$ by \citet{wang2020reliable}, which uses an optimal transport uncertainty set. As this estimator is tailored to MDPs with a discounted cost criterion, we set the corresponding discount factor to~$\gamma=85\%$ and normalize the estimator by~$(1-\gamma)^{-1}$ to approximate the long-run average cost of~$\pi$ as proposed by~\citet{tsitsiklis2002average}.

The goal of the first experiment is to empirically validate the statistical optimality of $V_\rho^\pi(\qhat)$. To this end, we fix a sample size~$T\in\{500, 1{,}000, 2{,}000\}$ and sample~$20$ independent state-action trajectories of length~$T$ from the Markov chain induced by the transition kernel~$Q$ and the behavioral policy~$\pi_0$. For each trajectory, we then compute the three distinct estimators. The empirical out-of-sample disappointment~$\widehat\beta$ of any estimator is defined as the percentage of the trajectories for which the estimator falls below the true expected long-run average reward~$V(f_\pi(\xi_0))$ of the evaluation policy~$\pi$. The out-of-sample disappointment of the two distributionally robust estimators can be tuned by changing the radii of the underlying uncertainty sets, while that of the marginalized importance sampling estimator can be tuned by applying an additive offset. All hyperparameter values that were tested in the first experiment are reported in Table~\ref{tab:params}.

\begin{table}[h!]
\centering
\caption{Tested hyperparameter values (uncertainty radii of the two distributionally robust estimators and additive offsets of the marginalized importance sampling estimator), where $\mc K=\{0,\ldots,9\}$}
{\small
\begin{tabular}{r|r|r|r}
\toprule
$T$ & Uncertainty radii of $V_\rho^\pi(\qhat)$ & Uncertainty radii of $\widehat V^\pi_{\rm OT}(\qhat)$ & Offsets of $\widehat V^\pi_{\rm MIS}(\qhat)$ \\
\midrule
$500$    & $\{0.01-1.111\times10^{-3}k: k\in\mc K\}$ & $\{5-0.44k: k\in\mc K\}$    & $\{0.2-0.02k: k\in\mc K\}$ \\
$1{,}000$ & $ \{0.01-1.111\times10^{-3}k:k\in\mc K\}$ & $\{2.5-0.17k:k\in\mc K\}$    & $\{0.15-0.011k:k\in\mc K\}$ \\
$2{,}000$ & $\{0.01 -1.111\times10^{-3}k:k\in\mc K\}$ & $\{2.5-0.17k:k\in\mc K\}$    & $\{0.15-0.016k:k\in\mc K\}$ \\
\bottomrule
\end{tabular}}
\label{tab:params}
\end{table}

Theorem~\ref{thm:pareto:vstar} ensures that the distributionally robust value estimator $V_\rho^\pi(\qhat)$ is the least conservative ({\em i.e.}, largest) of all estimators with a prescribed decay rate of the out-of-sample disappointment. For large finite values of~$T$ and for any choices of the hyperparameters that induce the same empirical out-of-sample disappointment, we thus expect our distributionally robust value estimator to exceed the two baseline estimators. This conjecture is supported by the numerical results in Figure~\ref{fig:OPE-plot}, which shows that our estimator Pareto dominates the two baselines in that it predicts the highest rewards for any given upper bound on the out-of-sample disappointment. The first experiment thus complements the theoretical analysis in Section~\ref{ssec:OPE:dist:shift} by showing that the statistical efficiency of the proposed estimator persists even for (practically relevant) finite sample sizes and fixed out-of-sample disappointment levels. 
We also emphasize that, while the two benchmark estimators require explicit knowledge of the behavioral policy~$\pi_0$, the proposed distributionally robust estimator can be evaluated without this knowledge.

\begin{figure}[h!]
    \centering
    \scalebox{0.9}{\begin{tikzpicture}

\begin{groupplot}[group style = {group size = 3 by 1, horizontal sep = 35}, width=2.5in,height=2in]

\nextgroupplot[
title={$T=500$},
xlabel={\(\widehat \beta\)},
xmin=-0.05,
xmax=0.7,
ymin=-2.2,
ymax=-1.55,
legend style = { column sep = 5pt, legend columns=-1, legend to name = groupa,}
]
    \addplot[only marks, tblue, mark=pentagon*,mark size=3pt] coordinates {
    (0.0, -1.6793291765168572)
(0.0, -1.6592282165770373)
(0.0, -1.6598238434468822)
(0.0, -1.6684882087634674)
(0.0, -1.6717336272935668)
(0.0, -1.6740053061712978)
(0.05, -1.662402540725568)
(0.0,-1.6632150683801457)
(0.0, -1.6806782910283602)
(0.05, -1.678829225639673)
    };
    \addlegendentry{$V^\pi_\rho(\qhat)$}
    
    \addplot[only marks, orange, mark=triangle*, mark size=3pt] coordinates {
    (0.05, -1.831025109435635)
(0.0, -1.7667784096417833)
(0.05, -1.7292386878530164)
(0.05, -1.7558456461952718)
(0.1, -1.7188412696264195)
(0.05, -1.7107912204997295)
(0.25, -1.6983549250308312)
(0.35, -1.6664278975394697)
(0.65, -1.6867433530862077)
(0.5, -1.641595319759397)
    };
    \addlegendentry{$\widehat{V}^\pi_{\rm MIS}(\qhat)+\Delta$}
    
    \addplot[only marks, fgreen, mark=diamond*, mark size=3pt] coordinates {
    (0.0, -2.144487059379989)
(0.05, -2.247086773018734)
(0.15, -2.130047393864639)
(0.15, -2.087149432271644)
(0.05, -2.080539427937026)
(0.0, -2.1356525825960064)
(0.1, -2.145492079143671)
(0.05, -2.154952766495558)
(0.0, -2.029628712356864)
(0.0, -2.169288345634899)
    };
    \addlegendentry{$\widehat{V}^\pi_{\rm OT}(\qhat)$}

\nextgroupplot[
title={$T=1{,}000$},
xlabel={\(\widehat \beta\)},
xmin=-0.05,
xmax=0.7,
ymin=-2.2,
ymax=-1.55,
]

    \addplot[only marks, tblue, mark=pentagon*,mark size=3pt, forget plot] coordinates {
      (0, -1.66208011)
      (0.0 , -1.6770026121556108)
(0.0 , -1.6979139900888143)
(0.0, -1.6897820087626265)
(0.0 , -1.6610096970045647)
(0.0 , -1.664708458652371)
(0.0 , -1.667075998044755)
(0.1 , -1.6776264255244808)
(0.05 , -1.6572088108895433)
(0.0 , -1.6723923527522029)
    };
    
    \addplot[only marks, orange, mark=triangle*, mark size=3pt, forget plot] coordinates {
      (0.35, -1.66402819)
      (0.0 , -1.836003816485618)
(0.0 , -1.799456208130301)
(0.05 , -1.8642684526598021)
(0.1 , -1.7559264575349527)
(0.1 , -1.7782729878636032)
(0.15 , -1.7469927672317027)
(0.1 , -1.7413430813603532)
(0.4 , -1.704681440455691)
(0.25 , -1.689524414466216)
    };
    
    \addplot[only marks, fgreen, mark=diamond*, mark size=3pt, forget plot] coordinates {
      (0, -1.9708738)
      (0.2 , -1.7615386854562118)
(0.15 , -1.8362728287768082)
(0.15 , -1.8854562758402431)
(0.25 , -1.687509351780097)
(0.15 , -1.8947242013004917)
(0.15 , -1.7458508233844487)
(0.1 , -1.9034083466559948)
(0.25 , -1.7783404377165077)
(0.05 , -1.8621828268424827)
    };
    
\nextgroupplot[
title={$T=2{,}000$},
xlabel={\(\widehat \beta\)},
xmin=-0.05,
xmax=0.7,
ymin=-2.2,
ymax=-1.55,
]              
    
    \addplot[only marks, tblue, mark=pentagon*,mark size=3pt, forget plot] coordinates {
    (0, -1.66240006)
    (0, -1.6951574)
    (0, -1.65109963)
    (0, -1.67557959)
    (0, -1.68483436)
    (0, -1.67002794)
    (0, -1.69545889)
    (0, -1.67668793)
    (0.05, -1.65357863)
    (0.05, -1.66347218)
    };
    
    \addplot[only marks, orange, mark=triangle*, mark size=3pt, forget plot] coordinates {
    (0.0, -1.7573594208466663)
(0.0, -1.8125978720269251)
(0.0, -1.7648925768258175)
(0.0, -1.7217588362626077)
(0.05, -1.7200925493553108)
(0.15, -1.75399960421312)
(0.3, -1.7291244822744507)
(0.1, -1.7468593651051598)
(0.3, -1.6596874778425843)
(0.2, -1.665798495683066)
    };
    
    \addplot[only marks, fgreen, mark=diamond*, mark size=3pt, forget plot] coordinates {
    (0.0, -1.9444780314297077)
(0.05, -1.9488549323301307)
(0.0, -1.8409822784464438)
(0.0, -1.8578202554092957)
(0.2, -1.7924294986799336)
(0.05, -1.8727554527745767)
(0.05, -1.8462306421848447)
(0.1, -1.8191896870786421)
(0.05, -1.863027953497489)
(0.05, -1.8306339624006653)
    };
\end{groupplot}
\node at ($(group c2r1) + (0,-3.3cm)$) {\ref{groupa}}; 

\end{tikzpicture}}    
    \caption{Scatter plot of the average reward predicted by different estimators against the empirical out-of-sample disappointment $\widehat\beta$. Points correspond to different hyperparameter values from Table~\ref{tab:params}.}
    \label{fig:OPE-plot}
\end{figure}
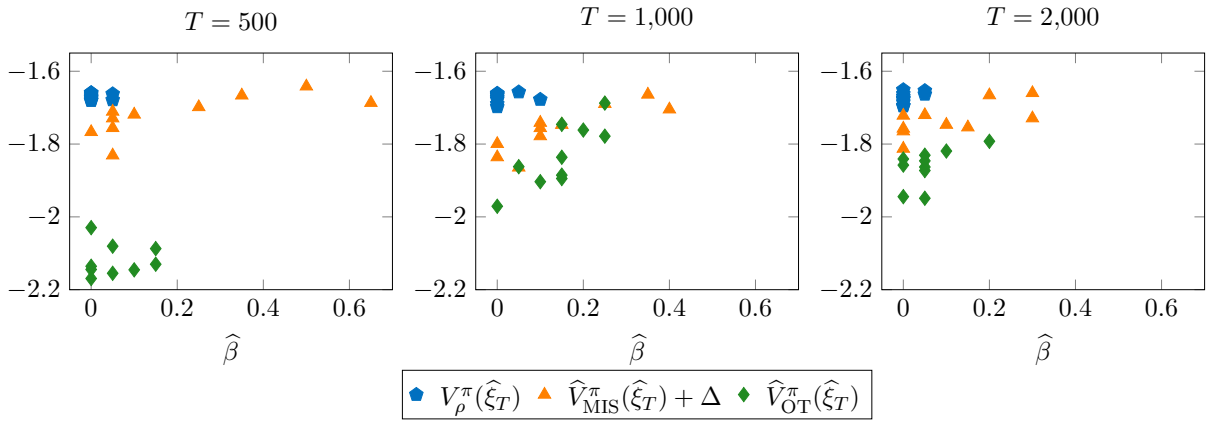


\subsection{Offline Policy Optimization: Machine Replacement}

The second experiment is based on the machine replacement problem described in~\citep{wiesemann2013robust}. The objective is to determine an optimal repair strategy for a machine whose condition is represented by eight ``operative'' states, labeled \(1, \dots, 8\), and two ``repair'' states, labeled~R1 and~R2. The available actions are either ``do nothing'' or ``repair.'' There is no reward in any of the operative states, while states~8, R1, and~R2 yield rewards of $-20$, $-2$, and $-10$ per time period, respectively. The initial state \(s_0\) is governed by the uniform distribution over \(\mathcal{S}\), and the transition kernel \(Q\) is defined as in~\cite[\S~6]{wiesemann2013robust}. 
Details are omitted for brevity.

In this experiment we compare the out-of-sample performance of our distributionally robust policy estimator~$\pi_\rho(\qhat)$ against that of two benchmark estimators. The first benchmark, denoted by~$\widehat{\pi}_{\rm KL}(\qhat)$, is obtained from a robust offline reinforcement learning problem with an $(s,a)$-rectangular uncertainty set defined using the Kullback-Leibler divergence \citep[\S~4]{shi2022distributionally}. This estimator enjoys near-optimal sample complexity. The second benchmark, denoted by~$\widehat{\pi}_{\rm PI}(\qhat)$, is based on a non-robust, model-based (``plug-in'') approach, which achieves minimax-optimal sample complexity \citep{li2024settling}. Both baseline estimators are designed to maximize discounted reward. Since the average reward of a policy is the leading term in the expansion of the discounted reward as the discount factor approaches~$1$, a policy that is optimal for large 
discount factors must be nearly optimal for the average reward criterion \citep[\S~10.1.2]{puterman2005markov}. We thus set the discount factor to~$0.95$ when computing the baseline estimators. We also emphasize that both baseline estimators require access to independent samples from the stationary state-action-next-state distribution. 
In contrast, the proposed distributionally robust policy estimator only requires access to one single state-action trajectory generated under an unknown behavioral policy. To ensure a fair comparison, we construct the baseline estimators from the~$T-1$ (dependent) state-action-next-state triples in the state-action trajectory of length~$T$ that is made available to all estimators. Finally, we set the radii of the uncertainty sets of the distributionally robust estimators $\widehat{\pi}_{\rm KL}(\qhat)$ and $\widehat{\pi}_{\rm PI}(\qhat)$ to~$4.5/T$. This scaling is recommended by \citet{duchi2021statistics} for $\widehat{\pi}_{\rm KL}(\qhat)$ and ensures that the out-of-sample disappointment of~$\pi_\rho(\qhat)$ remains approximately constant at $\beta\approx 1\%$ (see Theorem~\ref{thm:gen:bound:pi}).

We assume now that the behavioral policy $\pib \in \Pi_0$ selects each action $a\in\mc A$ with probability~$1/A$ irrespective of the current state. For any fixed sample size~$T$, we first generate a state-action trajectory of length~$T$ from the Markov chain induced by~$\pi_0$ and~$Q$.
For each of these trajectories, we then construct the three policy estimators and compute their true long-run average rewards. Finally, we record the frequency with which each estimator achieves the highest reward across the~$100$ simulation runs; see Figure~\ref{fig:oos-policy-opt}. We observe that our distributionally robust policy estimator wins most often for all sample sizes~$T \lesssim 400$. For larger sample sizes, the minimax-optimal plug-in estimator dominates. 
The robust policy estimator $\widehat{\pi}_{\rm KL}(\qhat)$ displays a similar performance as~$\pi_\rho(\qhat)$ for small values of~$T$. 
We believe that the performance of~$\pi_\rho(\qhat)$ deteriorates in this regime because of Algorithm~\ref{alg:PLD}, which outputs suboptimal reward estimates when the uncertainty set is large.

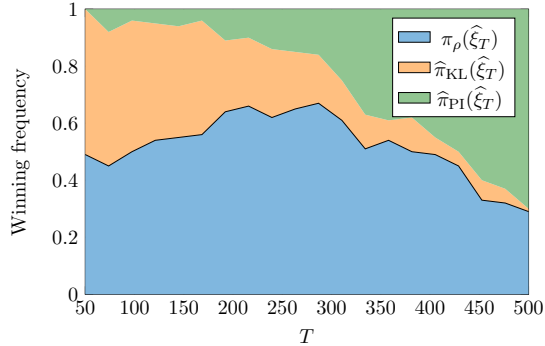
\begin{figure}[htb!]
    \centering
    \scalebox{.7}{\begin{tikzpicture}
\begin{axis}[
    width=10cm,
    height=7cm,
    xlabel={$T$},
    ylabel={Winning frequency},
    legend pos=north east,
    xmin=50, xmax=500,
    ymin=0, ymax=1,
    xtick={50,100,150,200,250,300,350,400,450,500},
    ytick={0,0.2,0.4,0.6,0.8,1.0},
  ]
    \path[name path=baseline] (axis cs:50,0) -- (axis cs:500,0);
    
        \addplot[name path=ours,  forget plot] coordinates {
      (50,0.49)
      (73.68,0.45)
      (97.37,0.5)
      (121.05,0.54)
      (144.74,0.55)
      (168.42,0.56)
      (192.11,0.64)
      (215.79,0.66)
      (239.47,0.62)
      (263.16,0.65)
      (286.84,0.67)
      (310.53,0.61)
      (334.21,0.51)
      (357.89,0.54)
      (381.58,0.5)
      (405.26,0.49)
      (428.95,0.45)
      (452.63,0.33)
      (476.32,0.32)
      (500,0.29)
    };
    
    \addplot[name path=cum1, draw=none, forget plot] coordinates {
      (50,1.0)        
      (73.68,0.92)    
      (97.37,0.96)    
      (121.05,0.95)   
      (144.74,0.94)   
      (168.42,0.96)   
      (192.11,0.89)   
      (215.79,0.90)   
      (239.47,0.86)   
      (263.16,0.85)   
      (286.84,0.84)   
      (310.53,0.75)   
      (334.21,0.63)   
      (357.89,0.61)   
      (381.58,0.62)   
      (405.26,0.55)   
      (428.95,0.50)   
      (452.63,0.40)   
      (476.32,0.37)   
      (500,0.30)      
    };
    
    \addplot[name path=cum2, draw=none, forget plot] coordinates {
      (50,1.0)
      (73.68,1.0)
      (97.37,1.0)
      (121.05,1.0)
      (144.74,1.0)
      (168.42,1.0)
      (192.11,1.0)
      (215.79,1.0)
      (239.47,1.0)
      (263.16,1.0)
      (286.84,1.0)
      (310.53,1.0)
      (334.21,1.0)
      (357.89,1.0)
      (381.58,1.0)
      (405.26,1.0)
      (428.95,1.0)
      (452.63,1.0)
      (476.32,1.0)
      (500,1.0)
    };

\addplot[tblue!50,forget plot] fill between[of=baseline and ours];

\addplot[orange!50,forget plot] fill between[of=ours and cum1];

\addplot[fgreen!50,forget plot] fill between[of=cum1 and cum2];

\addlegendimage{area legend, fill=tblue!50}
\addlegendentry{$\pi_\rho(\qhat)$}

\addlegendimage{area legend, fill=orange!50}
\addlegendentry{$\widehat{\pi}_{\rm KL}(\qhat)$}

\addlegendimage{area legend, fill=fgreen!50}
\addlegendentry{$\widehat{\pi}_{\rm PI}(\qhat)$}
\end{axis}
\end{tikzpicture}}
    \caption{Frequencies at which each of the three policy estimators achieves the highest long-run average reward across $100$ independent simulation runs, as a function of $T$.}
    \label{fig:oos-policy-opt}
\end{figure}

Figure~\ref{fig:oos-policy-opt} does not change substantially even if the baseline estimators are given unfair access to independent samples from the {\em true} stationary state-action-next-state distribution (not shown).

\section*{Acknowledgement}
This work was supported as a part of the NCCR Automation, a National Center of Competence in Research, funded by the Swiss National Science Foundation (grant number 51NF40\_225155).

\bibliography{names}


\end{document}